\documentclass[11pt,a4paper,reqno]{amsart}
\usepackage[a4paper,top=3cm,bottom=3cm,inner=3cm,outer=3cm,includehead,includefoot]{geometry}
\usepackage{amsmath,amssymb,amsthm,calc,verbatim,mathrsfs,paralist,tikz,calc,url}
\usepackage{enumitem}
\usetikzlibrary{shapes.misc,calc,intersections,patterns,decorations.markings,plotmarks}

\newtheorem{theorem}{Theorem}
\newtheorem{lemma}[theorem]{Lemma}

\newtheorem{corollary}[theorem]{Corollary}
\newtheorem{conjecture}[theorem]{Conjecture}

\newtheorem{question}[theorem]{Question}

\def\Int{\operatorname{Int}}
\def\CC{\mathbb{C}}
\def\EE{\mathbb{E}}
\def\HH{\mathbb{H}}
\def\NN{\mathbb{N}}
\def\PP{\mathbb{P}}
\def\RR{\mathbb{R}}

\def\XX{\mathcal{X}}
\def\ZZ{\mathbb{Z}}
\def\rr{{\mathcal r}}
\def\SS{{\mathcal S}}
\def\SSS{\Int{\mathcal S}}
\def\triD{\vec{\mathbb{T}}}
\def\ZZD{\vec{\mathbb{Z}}}
\def\tri{\mathbb{T}}
\def\UU{\mathcal{U}}
\def\tikzscale{1.2}
\def\latticescale{0.8}
\def\barrierscale{0.7}
\def\equivalenceScale{1.2}
\def\sqrtthree{1.73205080757}
\newcommand*\rows{6}
\newcommand*\halfrows{3}
\newcommand*\halfrowsminus{2}
\newcommand{\Bin}{\operatorname{Bin}}

\renewcommand{\leq}{\leqslant}
\renewcommand{\geq}{\geqslant}
\renewcommand{\to}{\rightarrow}

\def\dist{\mathop{{\rm dist}}\nolimits}

\newcommand{\closure}[1]{\left [ #1 \right ]}

\newenvironment{proofOfLemma}{\begin{proof}[\textit{Proof}]}{\end{proof}}
\newenvironment{proofOfDTBP}{\begin{proof}[\textit{Proof of the lower bound in Corollary \ref{cor:DTBP}}]}{\end{proof}}
\newenvironment{proofOfMain}{\begin{proof}[\textit{Proof of Theorem \ref{thm:subpc}}]}{\end{proof}}

\title[Subcritical $\UU$-bootstrap percolation]{Subcritical $\UU$-bootstrap percolation models have non-trivial phase transitions}

\author{Paul Balister}
\address{Department of Mathematical Sciences, University of Memphis, Memphis, Tennessee 38152, USA}
\email{pbalistr@memphis.edu}

\author{B{\'e}la Bollob{\'a}s}
\address{Department of Pure Mathematics and Mathematical Statistics, University of Cambridge, Wilberforce Road, Cambridge CB3 0WB, UK, and Department of Mathematical Sciences, University of Memphis, Memphis, Tennessee 38152, USA, and London Institute for Mathematical Sciences, 35a South St, Mayfair, London W1K 2XF, UK}
\email{b.bollobas@dpmms.cam.ac.uk}

\author{Micha{\l} Przykucki}
\address{Department of Pure Mathematics and Mathematical Statistics, University of Cambridge, Wilberforce Road, Cambridge CB3 0WB, UK, and London Institute for Mathematical Sciences, 35a South St, Mayfair, London W1K 2XF, UK}
\email{mp@lims.ac.uk}

\author{Paul Smith}
\address{IMPA, 110 Estrada Dona Castorina, Jardim Bot{\^a}nico, Rio de Janeiro, 22460-320, Brazil}
\email{psmith@impa.br}

\thanks{The second author is partially supported by NSF grant DMS~1301614 and MULTIPLEX no. 317532. The third author is supported by MULTIPLEX no. 317532. The fourth author is supported by a CNPq bolsa PDJ}



\begin{document}

\begin{abstract}

We prove that there exist natural generalizations of the classical bootstrap percolation model on $\ZZ^2$ that have non-trivial critical probabilities, and moreover we characterize all homogeneous, local, monotone models with this property.

Van Enter \cite{strayleysargument} (in the case $d=r=2$) and Schonmann \cite{cellularbehaviour} (for all $d \geq r \geq 2$) proved that $r$-neighbour bootstrap percolation models have trivial critical probabilities on $\ZZ^d$ for every choice of the parameters $d \geq r \geq 2$: that is, an initial set of density $p$ almost surely percolates $\ZZ^d$ for every $p>0$. These results effectively ended the study of bootstrap percolation on infinite lattices.

Recently Bollob\'as, Smith and Uzzell \cite{neighbourhoodBootstrap} introduced a broad class of percolation models called $\UU$-bootstrap percolation, which includes $r$-neighbour bootstrap percolation as a special case. They divided two-dimensional $\UU$-bootstrap percolation models into three classes -- subcritical, critical and supercritical -- and they proved that, like classical 2-neighbour bootstrap percolation, critical and supercritical $\UU$-bootstrap percolation models have trivial critical probabilities on $\ZZ^2$. They left open the question as to what happens in the case of subcritical families. In this paper we answer that question: we show that every subcritical $\UU$-bootstrap percolation model has a non-trivial critical probability on $\ZZ^2$. This is new except for a certain `degenerate' subclass of symmetric models that can be coupled from below with oriented site percolation. Our results re-open the study of critical probabilities in bootstrap percolation on infinite lattices, and they allow one to ask many questions of subcritical bootstrap percolation models that are typically asked of site or bond percolation.

\end{abstract}

\maketitle

\section{Introduction}

\subsection{Bootstrap percolation on infinite lattices}

The classical $r$-neighbour bootstrap percolation model was introduced by Chalupa, Leath and Reich \cite{bootstrapbethe} in order to model certain physical interacting particle systems. Given a graph $G=(V,E)$, usually taken to be $\ZZ^d$ or $[n]^d$, a subset $A \subset V$ of the set of vertices of $G$ is chosen by including vertices independently at random with probability $p$. We write $A \sim \Bin(V,p)$ to denote that the set $A$ has this distribution and $\PP_p$ for the product probability measure. The vertices in $A$ are said to be \textit{infected}. Set $A_0 = A$ and then, for $t=0,1,2,\ldots$, let
\[
A_{t+1} = A_{t} \cup \big\{ v \in V \, : \, |N(v)\cap A_t|\geq r \big\},
\]
where $N(v)$ is the set of neighbours of $v$ in $G$. Thus, infected vertices remain infected forever, and uninfected vertices become infected when at least $r$ of their neighbours in $G$ are infected. The \emph{closure} of $A$ is the set $[A]=\bigcup_{t=0}^\infty A_t$ of all vertices that are eventually infected. When $[A]=V$ we say that \emph{$A$ percolates $G$}, or simply that \emph{$A$ percolates}. We say that $A$ is \textit{closed under percolation} if $\closure{A} = A$.


One would like to know under what conditions on $G$ and $p$ it is likely that $A$ percolates $G$, so it is natural to define the \emph{critical probability} $p_c(G,r)$ by
\begin{equation}
\label{eq:p_c}
p_c(G,r) = \inf \{ p \, : \, \PP_p([A]=V(G)) \geq 1/2 \}.
\end{equation}
In the case $G=\ZZ^d$, by ergodicity (since the event that $A$ percolates $G$ is translation invariant), the probability that $A$ percolates $G$ is either $0$ or $1$. Hence, on $G=\ZZ^d$, in equation \eqref{eq:p_c} it is more natural to consider $\PP_p([A]=\ZZ^d) = 1$ instead of $\PP_p([A]=\ZZ^d) \geq 1/2$.

The first result in the field of bootstrap percolation was due to van Enter \cite{strayleysargument}, who proved in the case $d=r=2$ that for every positive initial density $p$ there is percolation almost surely, and hence that $p_c(\ZZ^2,2)=0$. This was later greatly generalized by Schonmann \cite{cellularbehaviour}, who showed that
\[
p_c(\ZZ^d,r) = \begin{cases} 0 & \text{if } 1\leq r\leq d, \\ 1 & \text{if } d+1\leq r\leq 2d. \end{cases}
\]
(The cases $r=1$ and $d+1\leq r\leq 2d$ are trivial; the content of the theorem is the assertion when $2\leq r\leq d$.)


The results of van Enter and Schonmann to a large extent ended the study of bootstrap percolation on infinite lattices. However, Aizenman and Lebowitz \cite{metastabilityeffects} recognized that bootstrap percolation exhibited interesting finite-size effects: on finite grids $[n]^d$, there is a certain metastability threshold for the initial density $p$, below which with high probability there is no percolation, and above which with high probability there is percolation. More precisely, Aizenman and Lebowitz showed that $p_c([n]^d,2)=\Theta\big((\log n)^{-(d-1)}\big)$. Holroyd \cite{sharpmetastability} later proved that $p_c([n]^2,2)=(1+o(1))\pi^2/18\log n$, and Gravner, Holroyd and Morris \cite{sharperThreshold} and Morris \cite{secondTerm} obtained bounds on the second order term. Cerf and Cirillo \cite{scalingthree} ($d=r=3$) and Cerf and Manzo \cite{regimefinite} ($d \geq r \geq 3$) determined $p_c([n]^d,r)$ up to a constant for all $r \geq 3$, and Balogh, Bollob{\'a}s and Morris \cite{bootstrapthree} ($d=r=3$) and Balogh, Bollob{\'a}s, Duminil-Copin and Morris \cite{sharpbootstrapall} ($d \geq r \geq 3$) determined the constant for all $r \geq 3$.

Returning to infinite lattices, except for a small number of degenerate examples, which we discuss in Section~\ref{sec:degenerate}, all of the bootstrap percolation models on $\ZZ^d$ and other lattices that have so far been studied have been shown to have critical probabilities on the appropriate infinite lattice equal to either $0$ or $1$. These include the $r$-neighbour model on $\ZZ^d$, the $r$-neighbour model on general lattices embedded in $\ZZ^d$ studied by Gravner and Griffeath \cite{thresholdGrowth}, the Duarte model studied by Schonmann \cite{criticalPoints} and Mountford \cite{mountfordDuarte}, and numerous other models (see, for example, \cite{BringMahl,vEH,HLR}). In a recent paper, Bollob\'as, Smith and Uzzell \cite{neighbourhoodBootstrap} introduced a new class of percolation models, called \emph{$\UU$-bootstrap percolation}, which contains bootstrap percolation as a special case. They showed that many $\UU$-bootstrap percolation models on $\ZZ^2$ (those which they termed \emph{supercritical} or \emph{critical}) also have critical probabilities equal to zero. They also conjectured that the remaining models (those which they termed \emph{subcritical}) have strictly positive critical probabilities. In this paper we prove this conjecture. Together with the results in \cite{neighbourhoodBootstrap}, this gives a complete characterization of bootstrap-like models on $\ZZ^2$ that have non-trivial critical probabilities, under some natural assumptions listed in the next subsection.

\subsection{$\UU$-bootstrap percolation}
\label{sec:gbpDefinition}

Under $\UU$-bootstrap percolation, new infections are made according to any rule that is local (the rule depends on a bounded neighbourhood of the vertex), homogeneous (the same rule applies to every vertex) and monotone (the set of neighbourhoods that infect a given site is an up-set). The formal definition is as follows. Let $\UU=\{X_1,\dots,X_m\}$ be a finite collection of finite, non-empty subsets of $\ZZ^d\setminus\{0\}$ and let $A=A_0\subset\ZZ^d$. Then for each $t\geq 0$, let
\[
A_{t+1} = A_{t} \cup \big\{ x\in\ZZ^d \, : \, \text{there exists } i\in[m] \text{ such that } X_i+x\subset A_t \big\}.
\]
The set $\UU$ is called an \emph{update family} and the sets $X_i$ \emph{update rules}. The $r$-neighbour model on $\ZZ^d$ is clearly an example of a $\UU$-bootstrap percolation model: it consists of $\binom{2d}{r}$ update rules, one for each $r$-subset of the neighbours of the origin. We again write $\closure{A}$ for the set of all vertices that eventually become infected, and say that $A$ is \textit{closed under $\UU$} if we have $\closure{A} = A$.

For the rest of the paper we shall restrict our attention to the case $d=2$. The rough behaviour of two-dimensional $\UU$-bootstrap percolation is determined by the action of the dynamics on discrete half planes. We use the notation $S^1$ for the unit circle in $\RR^2$ and for each $u \in S^1$ we let $\HH_u$ denote the discrete half plane $\{x\in\ZZ^2:\langle x,u\rangle<0\}$. An element $u \in S^1$ is said to be a \emph{stable direction} for the update family $\UU$ if $[\HH_u]=\HH_u$; that is, if no new sites become infected when the initial set is equal to the half plane $\HH_u$. Otherwise $u$ is said to be an \emph{unstable direction} for $\UU$. For every update family $\UU$ and every $u \in S^1$, the closure of $\HH_u$ is either $\HH_u$ or the whole plane $\ZZ^2$. The \emph{stable set} $\SS$ for $\UU$ is the set
\[
\SS = \SS(\UU) = \{u \in S^1 : \text{$u$ is stable for $\UU$}\}.
\]
We say that an update rule $X$ \textit{destabilizes} a direction $u \in S^1$ if for $\UU = \{X\}$ we have $u \notin \SS(\UU)$. One can easily show (see Theorem~1.10 of \cite{neighbourhoodBootstrap}) that a subset $\SS$ of the circle $S^1$ is the stable set of some update family $\UU$ if and only if $\SS$ can be expressed as a finite union of closed intervals in $S^1$ whose end-points have rational or infinite slope relative to the standard basis vectors.

Let ${\mathcal T} = \RR/2 \pi \ZZ$. We shall frequently need to change between elements of $S^1$ and elements of ${\mathcal T}$; in order to do this we define the natural bijection $u : {\mathcal T} \to S^1$ by $u(\theta) = (\cos \theta, \sin \theta)$, and we set $\theta = u^{-1}$ to be its inverse function.

We define the \textit{strongly stable set} $\SSS(\UU)$ for $\UU$ to be the interior of $\SS$, i.e.,
\[
\SSS = \SSS(\UU) = \{u \in S^1 : \exists \, \varepsilon > 0 \text{ such that if } |\theta(u) - \theta(v)| < \varepsilon \text{ then } v \in \SS\}.
\]
If $u \in \SSS$ then we say that $u$ is a \textit{strongly stable direction}. Clearly, any strongly stable direction is also a stable direction.

Bollob\'as, Smith and Uzzell divided $\UU$-bootstrap percolation models into three classes according to the structure of the stable set. They defined the update family $\UU$ to be:
\begin{enumerate}
\item \emph{supercritical} if there exists an open semicircle in $S^1$ that is disjoint from $\SS$; that is, if there do not exist three stable directions $u_1$, $u_2$ and $u_3$ such that the origin belongs to the interior of the triangle with vertices at $u_1$, $u_2$ and $u_3$;
\item \emph{critical} if every open semicircle in $S^1$ has non-empty intersection with $\SS$, but there exists a semicircle in $S^1$ that is disjoint from $\SSS$; that is, if there exist three stable directions $u_1$, $u_2$ and $u_3$ such that the origin belongs to the interior of the triangle with vertices at $u_1$, $u_2$ and $u_3$, but no such three strongly stable directions exist;
\item \emph{subcritical} if every open semicircle in $S^1$ has non-empty intersection with $\SSS$; that is, if there exist three strongly stable directions $u_1$, $u_2$ and $u_3$ such that the origin belongs to the interior of the triangle with vertices at $u_1$, $u_2$ and $u_3$.
\end{enumerate}
Analogously to $r$-neighbour bootstrap percolation, we define $p_c(\ZZ^2,\UU)$ to be the infimum of those values of $p$ for which percolation occurs almost surely under update family $\UU$. In \cite{neighbourhoodBootstrap} the authors show that if $\UU$ is either supercritical or critical then $p_c(\ZZ^2,\UU)=0$. In fact, they show considerably more: letting
\[
p_c(\ZZ^2,\UU,t) = \inf\big\{ p \, : \, \PP_p(0\in A_t) \geq 1/2 \big\},
\]
they show that $p_c(\ZZ^2,\UU,t)=t^{-\Theta(1)}$ when $\UU$ is supercritical and $p_c(\ZZ^2,\UU,t)=(\log t)^{-\Theta(1)}$ when $\UU$ is critical. (Considerably stronger results for critical models have since been proved by Bollob\'as, Duminil-Copin, Morris and Smith~\cite{BDCMS}.) They also conjecture that $p_c(\ZZ^2,\UU)>0$ when $\UU$ is subcritical. Here we prove that conjecture. The following is the main theorem of this paper.

\begin{theorem}\label{thm:subpc}
Let $\UU$ be a subcritical update family and let $A\sim\Bin(\ZZ^2,p)$. Then
\[
\PP_p\big(0\in[A]\big) \to 0 \quad \text{as} \quad p \to 0.
\]
In particular, $p_c(\ZZ^2,\UU) > 0$. Furthermore, $p_c(\ZZ^2,\UU)=1$ if and only if $\SS=S^1$.
\end{theorem}

The strength of Theorem~\ref{thm:subpc} lies in its generality: we prove that the critical probability is strictly positive for \emph{every} two-dimensional bootstrap-like model for which the critical probability has not already been shown to be equal to zero.

As previously remarked, Theorem~\ref{thm:subpc} was previously only known in a small number of exceptional cases, all of which we consider to be degenerate because they exhibit a certain symmetry property which trivializes the proof. We discuss these models further in Section~\ref{sec:degenerate}.

Combined with the results of~\cite{neighbourhoodBootstrap}, Theorem~\ref{thm:subpc} has the following corollary.

\begin{corollary}\label{co:iff}
Let $\UU$ be an update family. Then $p_c(\ZZ^2,\UU)>0$ if and only if $\UU$ is subcritical.
\end{corollary}

Thus, our main theorem allows us to characterize \emph{all} update families with non-trivial critical probabilities.


\subsection{The archetypal example: bootstrap percolation on the directed triangular lattice}
\label{sec:dtbp}

Let $\triD$ denote the triangular lattice embedded in $\CC$, oriented and scaled so that $0$ and $1$ are neighbouring vertices. Let the edges of the lattice be directed, for $k=0,1,2$, in the direction $e^{(2 k+1)\pi i/3}$. In the resulting directed graph $\triD=(V,E)$, edges around any given vertex alternate in-out. (See Figure \ref{fig:lattice}.)

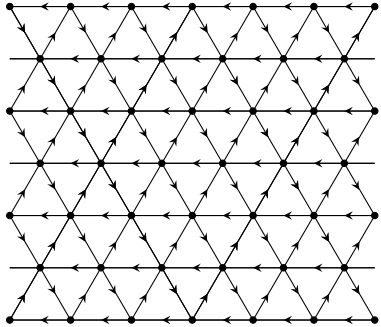
\begin{figure}[ht]
  \centering
\begin{tikzpicture}[scale=\latticescale,decoration={markings, mark=between positions 0 and 1 step \latticescale cm with {\arrow {stealth}}}]]
    \foreach \row in {0, ...,\rows} {
	\draw [] ($\row*(0, {0.5*sqrt(3)})$) -- ($(\rows,0)+\row*(0, {0.5*sqrt(3)})$);
	\draw [] ($\row*(0, {0.5*sqrt(3)})$) -- ($(\rows,0)+\row*(0, {0.5*sqrt(3)})$);
    }
    \foreach \row in {0, ...,\halfrows} {
	\draw [] ($\row*(1, 0)$) -- ($(\rows/2,{\rows/2*\sqrtthree})+\row*(1, 0)$);
	\draw [decorate] ($\row*(1, 0)+(1/4, {\sqrtthree/4})$) -- ($(\rows/2,{\rows/2*\sqrtthree})+(\row, 0)$);
	\draw [decorate] ($(\rows,0)+\row*(0, {\sqrtthree})-(1/2, 0)$) -- ($(0.5,0)+\row*(0, {\sqrtthree})$);
	\draw [] ($(\rows/2,{\rows/2*\sqrtthree})-\row*(1, 0)$) -- ($\row*(0, {\sqrtthree})$);
	\draw [] ($\row*(1, 0)$) -- ($\row*(0,\sqrtthree)$);
	\draw [decorate] ($(0,{\rows/2*\sqrtthree})+\row*(1,0)+(1/4, -{\sqrtthree/4})$) -- ($(\halfrows, 0)+\row*(1, 0)$);
	\draw [] ($(\halfrows, 0)+\row*(1, 0)$) -- ($(0,{\rows/2*\sqrtthree})+\row*(1,0)$);
	\foreach \site in {0, ...,\rows} {
		\filldraw ($2*\row*(0, {0.5*sqrt(3)})+\site*(1,0)$) circle (1.5pt);
	}
    }
    \foreach \row in {0, ...,\halfrowsminus} {
	\draw [decorate] ($(\rows,{\sqrtthree/2})+\row*(0, \sqrtthree)-(1, 0)$) -- ($(1,{\sqrtthree/2})+\row*(0, \sqrtthree)$);
	\draw [decorate] ($\row*(0, \sqrtthree)+(1/4, {\sqrtthree/4})$) --  ($(\halfrows,{\halfrows*\sqrtthree})-\row*(1, 0)$);
	\draw [decorate] ($(\halfrows,0)+\row*(1, 0)+(1/4, {\sqrtthree/4})$) -- ($(\rows,{\halfrows*\sqrtthree})-\row*(0, {\sqrtthree})$);
	\draw [] ($(\rows,{\halfrows*\sqrtthree})-\row*(0, {\sqrtthree})$) -- ($(\halfrows,0)+\row*(1, 0)$);
	\draw [decorate] ($(\halfrows,{\halfrows*\sqrtthree})+\row*(1,0)+(1/4,-{\sqrtthree/4})$) -- ($(\rows-1/4, {\sqrtthree/4})+\row*(0, \sqrtthree)+(1/4,-{\sqrtthree/4})$);
	\foreach \site in {1, ...,\rows} {
		\filldraw ($(-0.5,{\sqrtthree/2})+2*\row*(0, {0.5*sqrt(3)})+\site*(1,0)$) circle (1.5pt);
	}
    }
    \foreach \row in {1, ...,\halfrows} {
	\draw [decorate] ($\row*(0,\sqrtthree)+(1/4, -{\sqrtthree/4})$) -- ($\row*(1, 0)$);
	\draw [] ($(\rows, 0)+\row*(0, \sqrtthree)$) -- ($(\halfrows,{\halfrows*\sqrtthree})+\row*(1,0)$);
    }
\end{tikzpicture}
\caption{Directed triangular lattice $\triD$.}
\label{fig:lattice}
\end{figure}

Let $A_0=A\sim\Bin\big( V(\triD),p\big)$, and for each integer $t\geq 0$, define the set of infected sites at time $t+1$ to be
\[
A_{t+1} = A_{t} \cup \big\{ v\in V \, : \, |N^-(v)\cap A_t| \geq 2 \big\},
\]
where $N^-(v)$ is the set of in-neighbours of $v$ (that is, the set of vertices $u$ neighbouring $v$ such that $\overrightarrow{uv}$ is an edge). Note that $r=2$ is the only interesting value of the infection threshold for this model. We shall refer to this model as Directed Triangular Bootstrap Percolation (DTBP). It is easy to see by coupling that $p_c(\triD,2)$ is at most the critical probability for site percolation on $\tri$, the undirected triangular lattice, which is $p_c^s(\tri) = 1/2$. (See Theorem 17 in \cite{percolation}.) Indeed, by the uniqueness of the infinite cluster in percolation on $\tri$, if we initially infect the vertices of $\triD$ with probability $p \geq p_c^s(\tri)$ then almost surely all initially healthy clusters of sites will be finite, and any such region is eventually infected by the dynamics. However, it is not obvious whether $p_c(\triD,2 )$ is strictly positive. It is known that $p_c(\tri,3)=0$ (see, e.g., \cite{thresholdGrowth}) but there is no apparent coupling between the two models that we could use to deduce anything about the critical probability in the $2$-neighbour bootstrap process on $\triD$.

However, by skewing the lattice $\triD$, one can see that DTBP is equivalent to $\UU$-bootstrap percolation with update family $\UU_1 = \{X_1,X_2,X_3\}$, where $X_1 = \{(1,0),(0,1)\}$, $X_2 = \{(-1,-1),(0,1)\}$ and $X_3 = \{(-1,-1),(1,0)\}$. (See Figure \ref{fig:DTBPinZ2}.) Since $\UU_1$ is subcritical, Theorem \ref{thm:subpc} implies that
\[
0<p_c(\triD,2)<1.
\]
By analysing carefully the proof of Theorem \ref{thm:subpc}, one can in fact prove the following bounds for $p_c(\triD,2)$.

\begin{corollary}\label{cor:DTBP}
Under the DTBP subcritical $\UU$-bootstrap percolation model we have
\[
10^{-101} < p_c(\triD,2) = p_c(\ZZ^2,\UU_1) \leq 0.3118.
\]
\end{corollary}

The upper bound in Corollary \ref{cor:DTBP} is obtained by noting that DTBP can be coupled with oriented site percolation (see, for example, \cite{directedLower,directedUpper}). Indeed, $\UU$-bootstrap percolation with update family $\UU_2=\{X_1\}$ is precisely oriented site percolation: a site $v$ remains healthy forever if and only if there exists an infinite up/right path starting at $v$ of initially healthy sites.  The coupling with $\UU_2$ gives $p_c(\triD,2)\leq 1-p_c^s(\ZZD^2)\leq 0.3118$, where $p_c^s(\ZZD^2)$ is the critical probability for oriented site percolation on $\ZZ^2$, and the final inequality is due to Gray, Wierman and Smythe. For more information about percolation, see the book by Bollob{\'a}s and Riordan, \cite{percolation}.

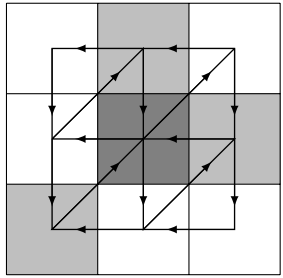
\begin{figure}[ht]
  \centering
  \begin{tikzpicture}[scale=\equivalenceScale]
    \fill[color=lightgray] (2,1) rectangle +(1,1);
    \fill[color=lightgray] (1,2) rectangle +(1,1);
    \fill[color=lightgray] (0,0) rectangle +(1,1);
    \fill[color=gray] (1,1) rectangle +(1,1);
    \draw[-,line width=0.5] (0.5,0.5) -- (0.5,2.5) -- (2.5,2.5) -- (2.5,0.5) -- (0.5,0.5) -- (2.5,2.5);
    \draw[-,line width=0.5] (0.5,1.5) -- (1.5,2.5) -- (1.5,0.5) -- (2.5,1.5) -- (0.5,1.5);
    \foreach \x/\y in {0.5/0.5, 1.5/1.5, 0.5/1.5, 1.5/0.5} {
        \draw [-latex] (\x,\y) to +(0.75,0.75);
    }
    \foreach \x/\y in {0.5/2.5, 0.5/1.5, 1.5/2.5, 1.5/1.5, 2.5/2.5, 2.5/1.5} {
        \draw [-latex] (\x,\y) to +(0,-0.75);
    }
    \foreach \x/\y in {2.5/0.5, 2.5/1.5, 2.5/2.5, 1.5/0.5, 1.5/1.5, 1.5/2.5} {
        \draw [-latex] (\x,\y) to +(-0.75,0);
    }
    \draw[step=1,black] (0,0) grid (3,3);
  \end{tikzpicture}
  \caption{The equivalence of the update family $\UU_1$ and the DTBP model; the dark grey site becomes infected when at least two of the light grey ones are.}
  \label{fig:DTBPinZ2}
\end{figure}

\begin{figure}[ht]
  \centering
  \begin{tikzpicture}[scale=\tikzscale]
    \draw[color=black] (0,0) circle (1);
    \draw [help lines,dashed,gray] (0,0) -- (1,0) arc (0:90:1) -- cycle;

    \draw [help lines,dashed,gray] (0,0) -- (-1,0) arc (180:135:1) -- cycle;

    \draw [help lines,dashed,gray] (0,0) -- (0,-1) arc (270:315:1) -- cycle;
    \draw [black,line width=2] (1,0) arc (0:90:1);
    \draw [black,line width=2] (-1,0) arc (180:135:1);
    \draw [black,line width=2] (0,-1) arc (270:315:1);
  \end{tikzpicture}
  \caption{The stable set $\SS_{1}$ for the update family $\UU_{1}$ (thick line); note that indeed every semicircle in $S^1$ intersects $\SSS_{1}$.}
  \label{fig:DTBPstableSet}
\end{figure}
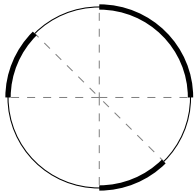

Computer experiments suggest that the true value of $p_c(\triD,2)$ is far from both the upper and lower bound in Corollary \ref{cor:DTBP}, indicating that in fact $p_c(\triD,2) \sim 0.118$. However, numerical predictions in bootstrap percolation have a long history of poor accuracy (see, e.g., \cite{sharpmetastability}), so this estimate should be taken with care.

\subsection{Symmetric models}\label{sec:degenerate}


Apart from oriented site percolation, other previously studied subcritical $\UU$-bootstrap percolation models include the model
\[\UU = \big\{\{(1,0),(0,1)\},\{(-1,0),(0,-1)\}\big\},\]
studied by Schonmann \cite{criticalPoints}; the knights, spiral and sandwich models, studied by Biroli and Toninelli \cite{glassTransition} and by Jeng and Schwarz \cite{jammingPercolation}; and the force-balance models, studied by Jeng and Schwarz \cite{forceBalance}. We would like to emphasize that none of these models is `typical' of the general model we study in this paper, in the following specific sense.

Let us say that a (necessarily subcritical) model $\UU$ is \emph{symmetric} if the following property holds: there exists $u\in S^1$ such that $\{u,-u\}\subset\Int\SS(\UU)$. It is easy to verify that all of the examples in the previous paragraph are symmetric. Now if $\UU$ is symmetric, then one can couple $\UU$-bootstrap percolation from below with oriented site percolation, which gives an essentially trivial proof of Theorem~\ref{thm:subpc} in the case of such models. We present this short and elementary proof in Section \ref{sec:percolationCoupling}.

In general, however, subcritical models need not be symmetric (DTBP is not symmetric, for example), and in these cases there does not seem to be a useful coupling with oriented site percolation. For such models, the lack of symmetry makes it considerably harder to control the growth of infected regions of sites, and the proof of Theorem~\ref{thm:subpc} is correspondingly more complex. Thus, the non-symmetric models are the ones that we consider to be `typical'.




\subsection{Organization of the paper}

The rest of this paper is organized as follows. In the next section we give an outline of the proof of Theorem~\ref{thm:subpc}, and we explain heuristically why one might expect the definition of a subcritical family to be the correct one. Following that, in Section~\ref{sec:notation}, we set out the standard notation we shall use, and we formalize some of the definitions relating to our construction. In Section \ref{sec:covers} we define and establish certain properties of ``barriers'' and ``triangular covers'', which will form the backbone of the coupled process we use in the proof of Theorem~\ref{thm:subpc}. In Section~\ref{sec:proof} we assemble the various tools from the previous sections in order to prove Theorem~\ref{thm:coversExist}, which is a certain statement about the existence of the ``triangular covers'', and which should be thought of as the heart of Theorem~\ref{thm:subpc}. We then deduce Theorem~\ref{thm:subpc} from Theorem~\ref{thm:coversExist}. We end the paper first with Section~\ref{sec:percolationCoupling}, in which we point out that Theorem~\ref{thm:subpc} is trivial if the update family $\UU$ is assumed to be symmetric, and second with Section~\ref{sec:openProblems}, in which we discuss a range of open problems and conjectures.

\section{Outline of the proof}
\label{sec:outline}

We know that supercritical and critical families have critical probability in $\ZZ^2$ equal to $0$, so what is special about subcritical families that makes them behave differently? Let $A$ be an initial set consisting of a rectangle of width $m$ and arbitrary height, and a density $p$ of sites above the rectangle. Under the classical two-neighbour bootstrap process on $\ZZ^2$ (which in a certain sense is representative of the behaviour of all critical processes), the infection spreads upwards from the rectangle, filling every line completely until it meets a fully healthy double line. The expected number of full new rows infected in the process is about $(1-p)^{-2m}$. The key property here is that a single site just above a full row will infect all other sites on the same row. In other words, if $R$ is the rectangle and $x$ a site next to its upper edge then under the two-neighbour process there is no upper bound on $|[R\cup\{x\}]|-|R|$ that is uniform in $m$.

Now consider the behaviour of the bootstrap process under an update family $\UU$ for which $u(\pi/2)$ is a strongly stable direction, that is, there is an interval of stable directions around $u(\pi/2)$. With the same $A$ as in the previous paragraph, how many new sites do we expect the process to infect? The key is that new sites create only localized infection: the set of additionally infected sites in the closure of the union of the rectangle and a small set $B$ of infected sites just above the top edge necessarily has ``small'' size, which depends on the size of $B$, on the stable set and on some additional characteristics of $\UU$, but not on the size of the rectangle. Given $B$ we can find a small circumscribed triangle $T$ of $B$, with sides of $T$ perpendicular to some stable directions within the interval of stable directions around $u(\pi/2)$. Assuming that $u(0)$, $u(\pi)$ and $u(3\pi/2)$ are also stable directions, if the slopes of $T$ are chosen appropriately to avoid the complications arising from the \textit{forbidden directions} which we define in Section \ref{sec:chooseDirections}, we have $[R \cup B] \subset [R \cup T] = R \cup T$.

The definition of a subcritical family is as follows: there exist three strongly stable directions $u_1$, $u_2$ and $u_3$ such that the origin belongs to the interior of the triangle with vertices at $u_1$, $u_2$ and $u_3$. Let $\HH_{u,a}$ denote the shifted half-plane $\{x\in\ZZ^2:\langle x-a,u\rangle<0\}$. Then the condition that the origin lies inside the triangle with vertices at $u_1$, $u_2$ and $u_3$ implies that the triangular sets of the form $\bigcap_{i=1}^3 \HH_{u_i,a_i}$, where the $a_i$ are arbitrary points in $\RR^2$, are necessarily finite. Also, we have $[\bigcap_{i=1}^3 \HH_{u_i,a_i}] = \bigcap_{i=1}^3 \HH_{u_i,a_i}$. In Section \ref{sec:chooseDirections} we show how to choose $u_1$, $u_2$ and $u_3$ so that these triangular sets are ``robust'' in the sense that they are still closed under $\UU$ if we slightly perturb their edges, making them a little bit ``wiggly''. This is quite unlike the two-neighbour process, where the only finite connected stable sets are rectangles, and new sites on their edges cause entire new rows or columns of infection. 

In our proof of Theorem \ref{thm:subpc} we exploit the above property of subcritical update families. We show that if every site in $\ZZ^2$ is initially infected independently with some probability $p > 0$ then, if $p$ is small enough, almost surely one can find a collection of slightly perturbed triangles (as above) with the following properties:
\begin{itemize}
\item every eventually infected site is contained in at least one triangle,
\item if two triangles have a nonempty intersection or, in fact, if they are not well separated, then one of them is contained in the other,
\item any site in $\ZZ^2$ belongs to at least one triangle with probability tending to $0$ as $p \to 0$. 
\end{itemize}
For sufficiently small $p$, the existence of a collection of triangular sets with these properties proves that the initial set does not percolate the plane, and this implies the lower bound on $p_c(\ZZ^2,\UU)$ in Theorem \ref{thm:subpc}.

We find our collection of perturbed triangles using a renormalization argument. Our method is motivated by the techniques introduced by G{\'a}cs \cite{clairvoyant} in the context of clairvoyant scheduling and a certain equivalent dependent oriented percolation model. We partition the plane using successively coarser tilings into squares of side lengths $\Delta_1 \ll \Delta_2 \ll \dots$. At each scale $\Delta_i$ we will have a notion of an \emph{$(i)$-good} $\Delta_i$-square, where ``good'' will roughly correlate with ``being sparsely infected'', and there will be a corresponding notion of an \emph{$(i)$-bad} $\Delta_i$-square. A little more precisely, a $\Delta_i$-square will be $(i)$-good if all $(i-1)$-bad $\Delta_{i-1}$-squares contained in it and in its close neighbourhood are quite strongly isolated.

Inductively we show that an $(i)$-bad $\Delta_i$-square contained in a $(i+1)$-good $\Delta_{i+1}$-square can be enclosed in a perturbed triangle which is not too large and is well separated, for all $j \leq i$, from all $(j)$-bad $\Delta_j$-squares which are not fully contained in it. Additionally, this perturbed triangle has sides essentially perpendicular to stable directions $u_1$, $u_2$ and $u_3$, i.e., is on its own closed under $\UU$. We do this by showing simultaneously by induction that, for any $i$, one can always find a ``thick'' healthy barrier through $(i)$-good $\Delta_i$-squares, disjoint from $(j)$-bad squares for all $j < i$. Since an $(i)$-bad $\Delta_i$-square contained in an $(i+1)$-good $\Delta_{i+1}$-square is necessarily surrounded by $(i)$-good $\Delta_i$-squares, this allows us to construct the triangular sets which enclose our eventually infected area.

The main task is the second part of the induction: to show that one can construct barriers through $(i)$-good $\Delta_i$-squares. The idea is that, since all $(i-1)$-bad $\Delta_{i-1}$-squares contained in an $(i)$-good $\Delta_i$-squares are quite strongly isolated, it is possible to ``navigate around'' these $(i-1)$-bad $\Delta_{i-1}$-squares without straying too far from a straight line, and to use the induction hypothesis to construct the barrier through the $(i)$-good $\Delta_i$-squares out of consecutive sub-barriers through $(i-1)$-good $\Delta_{i-1}$-squares.

In order to be a little more precise, suppose we are trying to construct a healthy barrier between sites $x$ and $y$, where these are such that the line $\ell$ joining them is roughly perpendicular to $u_1$ and only passes through $(i)$-good $\Delta_i$ squares. We shall show that there exist certain ``(i)-clean sites'' $c_1,\dots,c_k$, all of which lie close to $\ell$, such that the union of the lines joining $x$ to $c_1$, $c_1$ to $c_2$, and so forth, up to $c_k$ to $y$, only passes through $(i-1)$-good $\Delta_{i-1}$-squares. By induction, it follows that there exists a healthy barrier joining $x$ to $c_1$, etc., and one can show that it is possible to control these sufficiently such that their union is again a healthy barrier, but at the next scale. Thus, the edges of the perturbed triangles that we construct are in fact perturbed \emph{at all scales}.

This is the only part of the proof where we use the subcriticality of the update family and for that reason it is the most important part of our argument. The assertion that one can always find these perturbed triangles is Theorem~\ref{thm:coversExist}, and the (key) sub-assertion that one can always find these healthy barriers is Lemma~\ref{lem:barriersExist}: these two results should be regarded as the heart of Theorem~\ref{thm:subpc}.

\section{Additional notation and definitions}\label{sec:notation}

\subsection{Notation}

Given two sites $a, b \in \ZZ^2$ we define $\dist(a,b) = \|a-b\|_2$. For any two sets $A, B \subset \ZZ^2$ we then take
\[
 \dist (A,B) = \min_{a \in A, \, b \in B} \dist(a,b).
\]

For an update family $\UU$ we define
\[
 \rr(\UU) = \max_{i \in [m]} \: \max_{a,b \in X_i} \dist(a,b). 
\]
Hence, in particular, if $A$ is a set of initially infected sites such that any two distinct sites in $A$ are at distance larger than $\rr(\UU)$ then under update family $\UU$ we have $[A] = A$.

Given two sites $a, b \in \ZZ^2$, $a \neq b$, let
\[
 u_{a,b} = \frac{b-a}{\dist(a,b)} \in S^1.
\]

Subcritical update families are those for which there exist three strongly stable directions $u_1$, $u_2$ and $u_3$ such that the origin belongs to the interior of the triangle with vertices at $u_1$, $u_2$ and $u_3$. This can be rephrased as: there exist three distinct stable directions $u_1, u_2, u_3$ and positive numbers $\lambda_1, \lambda_2, \lambda_3, \varepsilon > 0$ such that
\begin{enumerate}
\item we have
\begin{equation}\label{eq:vecsum}
\lambda_1 u_1 + \lambda_2 u_2 + \lambda_3 u_3 = 0,
\end{equation}
\item for $t=1,2,3$,
\begin{equation}
\label{eq:uWithInters}
 \{u: |\theta(u_t)-\theta(u)| < \varepsilon \} \subset \SS.
\end{equation}
\end{enumerate}
To simplify our proof we will, somewhat counterintuitively, take the $\varepsilon$ in \eqref{eq:uWithInters} to be very small (which we are of course free to do).

\subsection{Choice of strongly stable directions and the first bound on $\varepsilon$}
\label{sec:chooseDirections}

In this section we choose our strongly stable directions $u_1$, $u_2$ and $u_3$, and we give a first upper bound on $\varepsilon$ in \eqref{eq:uWithInters}. The reason why we impose these particular conditions on our parameters will become clear in the proof of Lemma \ref{lem:bordersSurvive} in Section \ref{sec:covers}. Note that if $u_0$ is a strongly stable direction such that $N_\varepsilon(u_0) = \{u: |\theta(u_0)-\theta(u)| < \varepsilon \} \subset \SS$ then clearly also $N_\varepsilon(u_0) \subset \SSS$, i.e., all directions in $N_\varepsilon(u_0)$ are strongly stable. This means that the existence of one triple of strongly stable directions satisfying \eqref{eq:vecsum} implies the existence of infinitely many such triples.

Given an update family $\UU = \{X_1, \ldots, X_m\}$, we say that a direction $u \in S^1$ is \textit{forbidden for $\UU$} if it is perpendicular to at least one side of the convex hull of at least one of the update rules $X_i$ (note that every side of any convex hull forbids $2$ opposite directions). Let $F(\UU) = \{u: u \text{ is forbidden for } \UU\}$ be the set of directions forbidden for $\UU$. For example, for the update family $\UU_1$ equivalent to DTBP introduced in Section \ref{sec:dtbp} we have
\[
 \begin{split}
 F(\UU_1) = \Bigg \{ & \left (\frac{\sqrt{2}}{2}, \frac{\sqrt{2}}{2} \right ), \left (-\frac{\sqrt{2}}{2}, -\frac{\sqrt{2}}{2} \right ), \left (-\frac{2 \sqrt{5}}{5}, \frac{\sqrt{5}}{5} \right ), \\
						& \left (\frac{2 \sqrt{5}}{5}, -\frac{\sqrt{5}}{5} \right ), \left (-\frac{\sqrt{5}}{5}, \frac{2\sqrt{5}}{5} \right ), \left (\frac{\sqrt{5}}{5}, -\frac{2\sqrt{5}}{5} \right ) \Bigg \}.
 \end{split}
\]

Since $F(\UU)$ is a finite set, we can choose our strongly stable directions $u_1, u_2, u_3 \in \SSS(\UU) \setminus F(\UU)$ and let $\varepsilon (u_1, u_2, u_3)$ be small enough so that for $i = 1,2,3$, we have
\begin{equation}
\label{eq:avoidForbidden}
N_{\varepsilon(u_1, u_2, u_3)}(u_i) \subset \SSS \setminus F(\UU).
\end{equation}
To simplify our proof, from now on we assume that in \eqref{eq:uWithInters} we have $\varepsilon \leq \varepsilon(u_1, u_2, u_3)$.

\subsection{Good squares}
\label{sec:goodSquares}

Let us now define more precisely the tilings of $\ZZ^2$ we will work with in this paper, as well as the concepts of \textit{good} and \textit{bad} squares. The coarseness of our tilings and the definitions of good and bad squares will depend on the following parameters. Let
\begin{equation}
\label{eq:parameters}
1 < 1+\gamma < \beta < \alpha < 2 \qquad \text{and} \qquad \delta = (2\alpha+2\beta-3)/(2-\alpha).
\end{equation}
Let $\{\Delta_i\}_{i=1}^\infty$ be an increasing sequence of natural numbers with
\[
\Delta_{i+1} = \min \{n \in \NN: n \geq \Delta_i^{\alpha} \text{ and } n \text{ is a multiple of } \Delta_i \},
\]
with $\Delta_1 \geq \rr(\UU)$ to be defined later. For $i \geq 1$ let
\[
q_i = \Delta_i^{-\delta}, \qquad g_i = \Delta_i^{\beta}, \qquad \text{and} \qquad \sigma_i = \varepsilon/2 + \varepsilon \Delta_i^{-\gamma}.
\]
Note that, since $\alpha > \beta > 1$, we have $\Delta_{i+1} > g_i > \Delta_i$.

For each $i \geq 1$ let us consider an $(i)$-tiling of $\ZZ^2$ with $\Delta_i \times \Delta_i$ squares, i.e., a partition of $\ZZ^2$ into sets of the form
\[
\{a \Delta_i+1, a\Delta_i+2, \ldots, (a+1)\Delta_i\} \times \{b\Delta_i+1, b\Delta_i+2, \ldots, (b+1)\Delta_i\}
\]
for all $a, b \in \ZZ$. Note that our $(i)$-tilings are nested, i.e., that every $\Delta_{i+1} \times \Delta_{i+1}$ square consists of $\lceil \Delta_{i}^{\alpha-1} \rceil ^2$ squares of side length $\Delta_{i}$.

We shall define squares of side length $\Delta_i$ in our $(i)$-tiling of $\ZZ^2$ to be either \textit{$(i)$-good} or \textit{$(i)$-bad}. A $\Delta_1$-square is $(1)$-good if all its sites are initially healthy, otherwise it is $(1)$-bad. For $i \geq 1$ we declare a square $S$ of side length $\Delta_{i+1}$ to be $(i+1)$-bad if there exist two distinct non-adjacent squares (we consider squares that only touch corners as adjacent) $S', S''$ of side length $\Delta_{i}$ in our $(i)$-tiling (where $S'$ and $S''$ might be disjoint from $S$) which are $(i)$-bad and such that $\max \{ \dist(S,S') , \dist(S,S''), \dist(S',S'') \} \leq g_i$.

For $i \geq 1$ and an $(i)$-good square $S$ we say that a site $v \in S$ is \textit{$(i)$-clean} if, for all $j < i$, $v$ is at distance at least $g_j/3$ from any $(j)$-bad square.



\section{Barriers and triangular covers}
\label{sec:covers}

In this section we define barriers and triangular covers. We shall use these concepts in our proof to show that for $p > 0$ small enough the infection does not spread through the whole $\ZZ^2$, by showing that the closure of the initial infection can be enclosed in a collection of separated, finite sets of a special triangular shape.

Recall that we assume that for our update family $\UU$ we have
\begin{equation}
\label{eq:SS}
 \bigcup_{t=1}^3 \{u: |\theta(u_t)-\theta(u)| < \varepsilon \} \subset \SS.
\end{equation}
If for some $t \in \{1,2,3\}$ we have
\[
 \left | \big (\theta(u_{x,y}) - \theta(u_t) \big ) (\text{mod } 2\pi) - \pi/2 \right | < \sigma_1,
\]
(roughly speaking, if $u_{x,y}$ is ``nearly'' perpendicular to the stable direction $u_t$), then a \textit{$(1,t)$-barrier joining $x$ to $y$} is the set of all sites $v \in \ZZ^2$ such that for some $\lambda \in [0,1]$ we have
\[
\dist (v, \lambda x + (1-\lambda) y) \leq \rr(\UU).
\]
Let $i \geq 2$ and $x, y \in \ZZ$ be such that
\[
\left | \big (\theta(u_{x,y}) - \theta(u_t) \big ) (\text{mod } 2\pi) - \pi/2 \right | < \sigma_i.
\]
Let, for some $m \geq 1$, the sequence $(z_j)_{j=0}^m$ with $z_0 = x$, $z_m = y$ and $z_j \in \ZZ^2$ for all $j=1,2, \ldots, m-1$, be such that for all $j=1,2, \ldots, m$ we have
\[
\left | \left (\theta(u_{z_{j-1},z_j}) - \theta(u_t) \right ) (\text{mod } 2\pi) - \pi/2 \right | < \sigma_{i-1}.
\]
Then the set of all sites $v \in \ZZ^2$ such that for some $j \in \{1,2, \ldots, m\}$ and some $\lambda \in [0,1]$ we have
\[
 \dist (v, \lambda z_{j-1} + (1-\lambda) z_j) \leq \rr(\UU)
\]
is an \textit{$(i,t)$-barrier joining $x$ to $y$} (see Figure \ref{fig:(i,t)-barrier}). The sequence $(z_j)_{j=0}^m$ is called the \emph{anchor} of the $(i,t)$-barrier. Note that for $i \geq 2$, an $(i,t)$-barrier consists of $m \geq 1$ segments each of which is itself an $(i-1,t)$-barrier. This compound structure will allow $(i,t)$-barriers to avoid infected regions in $\ZZ^2$. We shall later use such infection-avoiding barriers and, exploiting the fact that they are essentially perpendicular to stable directions, enclose infected regions in hulls from which they cannot break out.
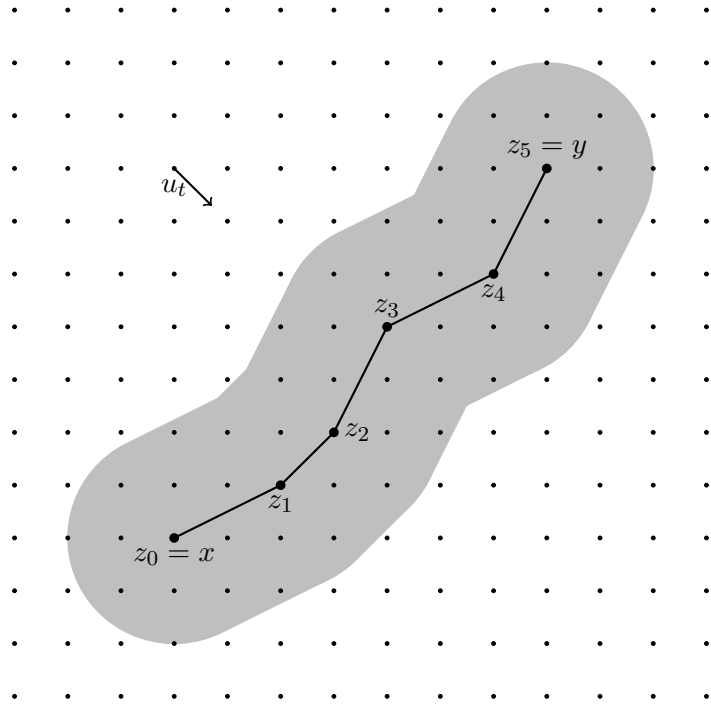
\begin{figure}[ht]
  \centering
  \begin{tikzpicture}[scale=\barrierscale]

	\draw [lightgray, fill=lightgray] (116.565:2) arc (116.565:296.565:2) -- ++(2,1) arc (-63.435:116.565:2) -- cycle;
	\draw [lightgray, fill=lightgray, xshift=2cm, yshift=1cm] (135:2) arc (135:315:2) -- ++(1,1) arc (-45:135:2) -- cycle;
	\draw [lightgray, fill=lightgray, xshift=3cm, yshift=2cm] (153.435:2) arc (153.435:333.435:2) -- ++(1,2) arc (-26.565:153.435:2) -- cycle;
	\draw [lightgray, fill=lightgray, xshift=4cm, yshift=4cm] (116.565:2) arc (116.565:296.565:2) -- ++(2,1) arc (-63.435:116.565:2) -- cycle;
	\draw [lightgray, fill=lightgray, xshift=6cm, yshift=5cm] (153.435:2) arc (153.435:333.435:2) -- ++(1,2) arc (-26.565:153.435:2) -- cycle;
	\foreach \x in {-3,-2,...,10}{
		\foreach \y in {-3,-2,...,10}{
			\node[draw,circle,inner sep=0.5pt,fill] at (\x,\y) {};
		}
	}
	\draw [->, thick] (0,7) node [below] {$u_t$} -- (0.7071,6.2929);
	\tikzstyle{every node}=[draw,circle,inner sep=1pt,fill]
	\draw (0,0) [thick] node {} -- (2,1) node {} -- (3,2) node {} -- (4,4) node {} -- (6,5) node {} -- (7,7) node {} ;
	\tikzstyle{every node}=[]
	\draw (0,0) node [below] {$z_0=x$} -- (2,1) node [below] {$z_1$} -- (3,2) node [right] {$z_2$} -- (4,4) node [above] {$z_3$} -- (6,5) node [below] {$z_4$} -- (7,7) node [above] {$z_5=y$};

  \end{tikzpicture}
  \caption{An example of an $(i,t)$-barrier joining $x$ to $y$ with $\rr(\UU) = 2$.}
  \label{fig:(i,t)-barrier}
\end{figure}

We may assume that $0 \leq \theta(u_1) < \theta(u_2) < \theta(u_3) < 2\pi$. Let $K \subset \ZZ^2$ be finite, let $i\geq 1$, and suppose $x, y, z \in \ZZ^2$ are distinct points such that:
\begin{itemize}
\item an $(i,1)$-barrier joining $x$ to $y$, an $(i,2)$-barrier joining $y$ to $z$ and an $(i,3)$-barrier joining $z$ to $x$ exist, and
\item $K$ lies inside the area bounded by these barriers and is disjoint from them.
\end{itemize}
Then we call the union $B$ of the three barriers an \emph{$(i)$-barrier cover} of $K$, and we call the union $T$ of $B$ and the sites in the area bounded by $B$ an \emph{$(i)$-triangular cover} of $K$. Note that there exist infinitely many $(i)$-barrier covers and infinite many $(i)$-triangular covers of any given finite set $K$ for every $i\geq 1$.

The $(i,t)$-barriers are perpendicular to strongly stable directions. In the next lemma we use this fact to show that for any finite set $K$, any $i \geq 1$, and any $(i)$-barrier cover $B$ and associated $(i)$-triangular cover $T$ of $K$, the closure $\closure{K}$ is a subset of $T \setminus B$ and is therefore isolated from $\ZZ^2 \setminus T$ by a barrier of thickness at least $\rr(\UU)$.

Note that for any subcritical update family $\UU=\{X_1,\dots,X_m\}$, for all $1 \leq i \leq m$ we have $|X_i| \geq 2$. Indeed if, without loss of generality, $X_1 = \{(x,y)\}$, then every direction $u \in S^1$ such that $\langle (x,y),u\rangle<0$ is an unstable direction. This set of directions constitutes an open semicircle in $S^1$ and hence the family $\UU$ is supercritical. Also, we then trivially have $p_c(\ZZ^2,\UU) = 0$: every site $(u,v) \in \ZZ^2$ will become infected if for some $t \geq 1$ the site $(u,v) + t \cdot (x,y)$ is initially infected and this happens almost surely for any $p > 0$.

\begin{lemma}
\label{lem:bordersSurvive}
Let $K\subset\ZZ^2$ be finite, let $i \geq 1$, and let $B$ be an $(i)$-barrier cover of $K$ and $T$ its associated $(i)$-triangular cover. Then
\[
\closure{K} \subset \closure{T \setminus B} = T \setminus B.
\]
\end{lemma}

\begin{proof}
The first containment $\closure{K} \subset \closure{T \setminus B}$ is obvious because $K \subset T\setminus B$. Therefore we only need to prove that $\closure{T\setminus B} = T\setminus B$, i.e., that $T \setminus B$ is closed under $\UU$.

Assume that the initial set of infected sites is $T \setminus B$. Recall that we have $u_1$, $u_2$, $u_3$ and $\varepsilon \leq \varepsilon(u_1, u_2, u_3)$ in Section~\ref{sec:chooseDirections} such that for $t \in \{1,2,3\}$ we have $N_{\varepsilon(u_1, u_2, u_3)}(u_t) \subset \SSS \setminus F(\UU)$. 

A site $v \in \ZZ^2 \setminus (T \setminus B)$ can become infected for three, essentially different, reasons. These are schematically shown in Figure \ref{fig:coversStable}, where we assume that $T \setminus B$ lies below the solid curve. Cases (1) and (2) in Figure \ref{fig:coversStable} correspond to $v$ being infected using update rules $X'$ and $X''$, which destabilize directions $u'$ and $u''$ respectively. For simplicity we assume $|X'|=|X''|=2$. Case (3) corresponds to $v$ being infected using an update rule $X'''$ that does not destabilize any directions. Rules of this type necessarily contain the origin in their (closed) convex hull; in the figure, for simplicity, we assume $|X'''|=3$.

\begin{figure}[ht]
  \centering
  \begin{tikzpicture}[scale=\barrierscale]

    \foreach \y in {0, 4, 8} {
      \draw[] (0,\y) to [out=35,in=140] (5,\y) to [out=320,in=205] (10,\y) to [out=25,in=140] (15,\y);
    }


    \draw [help lines, dashed] (3,9.68) -- (7,6.31);
    \draw (4,7.5) node {$\times$};
    \draw (3.5,8.5)  node {$\times$};
    \draw (7.5,9) [above] node {$v$};
    \node[draw,circle,inner sep=0,minimum size=0.1cm,fill] at (7.5,9) {};
    \draw [->, thick] (5.9749,7.181) -- (7.5,9) node [black,midway,right] {$u'$};;
    \draw (7.5,5.5) [above] node {(1)};


    \draw [help lines, dashed] (4,4.3) -- (10.5,3.8);
    \draw (4,4.3) node {$\times$};
    \draw (10.5,3.8)  node {$\times$};
    \draw (7,3.4) [left] node {$v$};
    \node[draw,circle,inner sep=0,minimum size=0.1cm,fill] at (7,3.4) {};
    \draw [->, thick] (7.0497,4.0459) -- (7,3.45) node [black,midway,right] {$u''$};;
    \draw (7.5,1.5) [above] node {(2)};


    \draw [help lines, dashed] (4,0.3) -- (10.5,-0.2) -- (6.5,-1.25) -- cycle;
    \draw (4,0.3) node {$\times$};
    \draw (10.5,-0.2)  node {$\times$};
    \draw (6.5,-1.25)  node {$\times$};
    \draw (7,-0.5) [left] node {$v$};
    \node[draw,circle,inner sep=0,minimum size=0.1cm,fill] at (7,-0.5) {};
    \draw (7.5,-2.5) [above] node {(3)};
  \end{tikzpicture}
  \caption{Three ways to infect a site $v \in \ZZ^2 \setminus (T \setminus B)$. The sites in $v+X'$, $v+X''$ and $v+X'''$ are denoted by $\times$.}
  \label{fig:coversStable}
\end{figure}
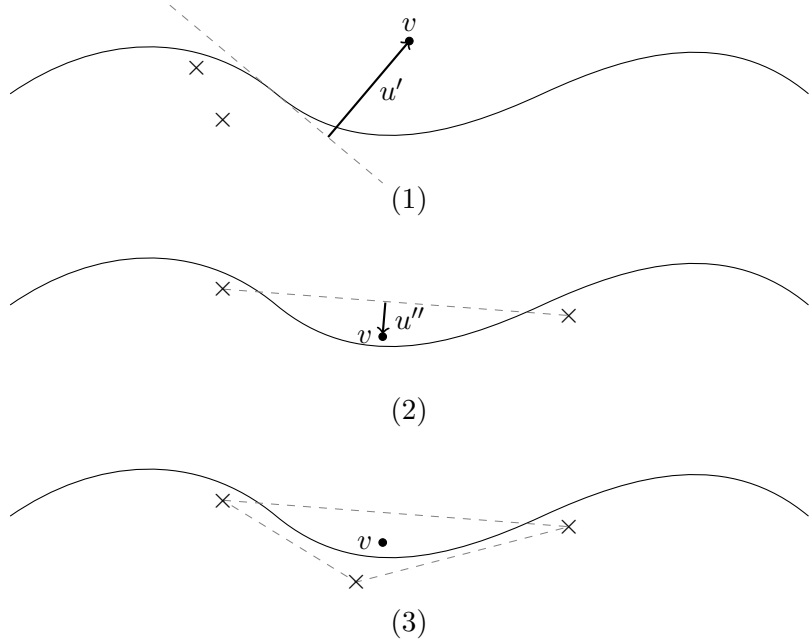

The site $v$ cannot be infected for the reason shown in case (1) of Figure \ref{fig:coversStable}, because the existence of such a rule $X' \in \UU$ would contradict the fact that for $t \in \{1,2,3\}$ we have $N_{\varepsilon(u_1, u_2, u_3)}(u_t) \subset \SS$. It cannot be infected for the reason shown in cases (2) or (3) of Figure \ref{fig:coversStable}, because now the existence of such a rule would contradict the fact that $N_{\varepsilon(u_1, u_2, u_3)}(u_t) \cap F(\UU) = \emptyset$. Hence $T \setminus B$ is closed under $\UU$, which completes the proof.
\end{proof}




In the next lemma we show that there exists a constant $c = c(\UU)$ such that for all $i \geq 1$ and all sufficiently large $\Delta$, we can find an $(i)$-triangular cover of a square of side length $\Delta$ in a ``small'' neighbourhood of that square, i.e., in a larger square of side length at most $c\Delta$.

\begin{lemma}
\label{lem:coverExists}
There exists $\ell_0\in\NN$ and $\varepsilon_0>0$ depending only on $\UU$ such that the following hold. Let $\varepsilon\leq\varepsilon_0$, $\ell\geq\ell_0$, $i\geq 1$, and $\Delta\geq\rr(\UU)$. Consider the tiling of $[c \Delta]^2$ consisting of $(2\ell+1)^2$ squares of side length $\Delta$, where $c=2\ell+1$. Then this tiling contains three distinct $\Delta \times \Delta$ squares $Y_1$, $Y_2$ and $Y_3$ such that for all $y_1 \in Y_1$, $y_2 \in Y_2$ and $y_3 \in Y_3$, and for each $t=1,2,3$, we have
\begin{equation}\label{eq:lemslope}
\left | \big (\theta(u_{y_t,y_{t+1}}) - \theta(u_t) \big ) (\operatorname{mod } 2\pi) - \pi/2 \right | < \varepsilon/2,
\end{equation}
where $y_4 = y_1$.

Additionally, every $(i,1)$-barrier joining $y_1$ to $y_2$, every $(i,2)$-barrier joining $y_2$ to $y_3$ and every $(i,3)$-barrier joining $y_3$ to $y_1$, is contained within the tiling and is disjoint from its middle square, i.e., from
\[
Y_0 = [\Delta \ell+1,\Delta(\ell+1)] \times [\Delta \ell+1,\Delta(\ell+1)].
\]
\end{lemma}

\begin{proof}
Let $i \geq 1$ and $\Delta\geq\rr(\UU)$, and let $\ell \geq 1$ be sufficiently large. Let
\[
v = (\Delta \ell+(\Delta+1)/2, \Delta \ell+(\Delta+1)/2)
\]
 be the midpoint of $Y_0$. The whole of $Y_0$ is clearly contained in a circle of radius $\Delta$ centered at $v$. For $r > 1$ to be specified later, let $S_1$ and $S_2$ be the circles centered at $v$ of radius $r \Delta$ and $(r+3) \Delta$ respectively. Also, let $T_1$ and $T_2$ be the triangles circumscribed on $S_1$ and $S_2$ respectively, tangent to these circles, for $t=1,2,3$, at points $v + r \Delta u_t$ and $v + (r+3) \Delta u_t$ respectively. (See Figure \ref{fig:coverExists}.) Independently of the values of $u_j$ and $r$, the three grey corner regions in Figure \ref{fig:coverExists} are each large enough to contain a disc of diameter $3 \Delta$, each of which itself contains a $\Delta \times \Delta$ square of the tiling of $[c \Delta]^2$. Fix any such three squares $Y_1$, $Y_2$ and $Y_3$. We claim that if $r$ is large enough and $\varepsilon > 0$ is small enough (both independently of $i$) then $Y_1$, $Y_2$ and $Y_3$ satisfy the conclusions of the lemma. 

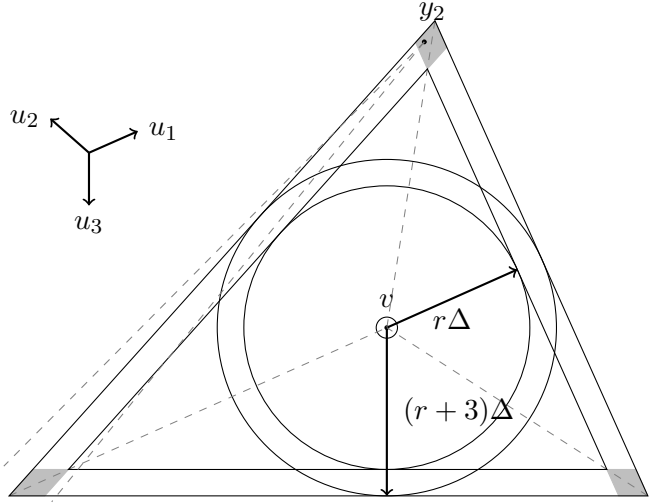
\begin{figure}[ht]
  \centering
  \begin{tikzpicture}[scale=\barrierscale]
		\draw [->, thick] (1.5,6.5) -- (2.4185,6.9082) node [right] {$u_1$};
		\draw [->, thick] (1.5,6.5) -- (0.7724,7.1468) node [left] {$u_2$};
		\draw [->, thick] (1.5,6.5) -- (1.5,5.5) node [below] {$u_3$};
		\draw (7.0964,3.1867) circle (3.1867);
		\draw (7.0964,3.1867) circle (2.6867);
		\draw (7.0964,3.1867) circle (0.2);
		\draw [->, thick] (7.0964,3.1867) -- (7.0964,0) node [midway,right=2pt] {$(r+3) \Delta$};
		\draw [->, thick] (7.0964,3.1867) -- (9.5473,4.2873) node [midway,below] {$r \Delta$};
		\draw (7.0964,3.1867) node [above=4pt] {$v$};
		\node [draw,circle,inner sep=0,minimum size=0.05cm,fill] at (7.0964,3.1867) {};
		\draw [fill,lightgray] (0,0) -- (0.4444,0.5) -- (1.1135,0.5) -- (0.669,0) -- cycle;
		\draw [fill,lightgray] (12,0) -- (11.4528,0) -- (11.2306,0.5) -- (11.7778,0.5) -- cycle;
		\draw [fill,lightgray] (8,9) -- (7.6352,8.5896) -- (7.8582,8.0879) -- (8.223,8.4982) -- cycle;
		\draw [help lines,dashed,gray] (0,0) -- (7.0964,3.1867);
		\draw [help lines,dashed,gray] (12,0) -- (7.0964,3.1867);
		\draw [help lines,dashed,gray] (8,9) -- (7.0964,3.1867);
		\draw (0,0) -- (12,0) -- (8,9) -- cycle;
		\draw (1.1135,0.5) -- (11.2306,0.5) -- (7.8582,8.0879) -- cycle;
		\draw (7.95,8.6) node [above=4pt] {$y_2$};
		\node [draw,circle,inner sep=0,minimum size=0.05cm,fill] at (7.8,8.6) {};
		\draw [dashed, gray] (7.8,8.6) -- (-0.1,0.5629);
		\draw [dashed, gray] (7.8,8.6) -- (0.8,-0.1129);
  \end{tikzpicture}
  \caption{Finding $(i,t)$-barriers in the neighbourhood of $v$.}
  \label{fig:coverExists}
\end{figure}

For $t=1,2,3$, let $\theta_t = \theta(u_t)$. Without loss of generality we may assume that, modulo $2\pi$, we have $\theta_3 - \theta_2 \geq \theta_2 - \theta_1 \geq \theta_1 - \theta_3$, and we may also assume that $r \geq 3$. The longest side of $T_2$ has length
\[
\begin{split}
 a_{max} & = (r+3) \Delta \left ( \tan \left ( \frac{\theta_3 - \theta_2}{2} \right ) +  \tan \left ( \frac{\theta_2 - \theta_1}{2} \right ) \right ) \\
	 & \leq 2r \Delta \left ( \tan \left ( \frac{\theta_3 - \theta_2}{2} \right ) +  \tan \left ( \frac{\theta_2 - \theta_1}{2} \right ) \right ),
\end{split}
\]
while the shortest side of $T_1$ has length 
\[
 a_{min} = r \Delta \left ( \tan \left ( \frac{\theta_2 - \theta_1}{2} \right ) +  \tan \left ( \frac{\theta_1 - \theta_3}{2} \right ) \right ).
\]

First we verify that \eqref{eq:lemslope} holds. Let $y_1 \in Y_1$, $y_2 \in Y_2$ and $y_3 \in Y_3$. For $t=1,2,3$, we must show that $\theta(u_{y_t,y_{t+1}})$ is at most $\varepsilon/2$ away from the angle of the vector perpendicular to $u_t$ (where again $y_4=y_1$). This holds if 
\begin{equation}
 \label{eq:a_min}
 a_{min} \tan \left ( \frac{\varepsilon}{2} \right ) \geq 3 \Delta,
\end{equation}
because this condition guarantees that the whole grey corner region containing $Y_{t+1}$ is contained inside the angle with its vertex at $y_t$ and of measure $\varepsilon$, lying symmetrically around the line perpendicular to $u_t$ which goes through $y_t$. (See Figure \ref{fig:coverExists} with $t=2$.) Inequality \eqref{eq:a_min} is satisfied whenever
\[
 r \geq 3 \left( \tan \left ( \frac{\varepsilon}{2} \right ) \right )^{-1} \Bigg ( \tan \left ( \frac{\theta_2 - \theta_1}{2} \right ) +  \tan \left ( \frac{\theta_1 - \theta_3}{2} \right ) \Bigg )^{-1} = r_\varepsilon.
\]
Thus, \eqref{eq:lemslope} holds provided $r\geq r_\varepsilon$.

Finally we must show that the condition in the last paragraph of the lemma holds. Given any two sites $u$ and $w$ in $\ZZ^2$, and $t \in \{1,2,3\}$, the sequence $(z_j)_{j=0}^m$ of points forming the anchor of an $(i,t)$-barrier joining $u$ to $w$ is, by the definition of an $(i,t)$-barrier, contained in a rhombus with two of its vertices at $u$ and $w$ and the interior angles at these two vertices equal to $2\varepsilon$. Now, if $u$ and $w$ are contained in different grey corner regions in Figure \ref{fig:coverExists}, then one can easily verify that this rhombus is at distance at least $r\Delta/6$ from the circle of radius $\Delta$ centered at $v$, provided $a_{max} \tan \varepsilon \leq r \Delta/2$, which holds if
\[
\begin{split}
\varepsilon \leq \arctan \left (\frac{1}{4} \Bigg (\tan \left ( \frac{\theta_3 - \theta_2}{2} \right ) +  \tan \left ( \frac{\theta_2 - \theta_1}{2} \right ) \Bigg )^{-1} \right ) = \varepsilon_0.
\end{split}
\]
Note that $\varepsilon_0$ depends on the values of $\theta_t$ only. Thus, to ensure that every $(i,t)$-barrier joining $u$ and $w$ is disjoint from the small circle centered at $v$, it is enough to have $r \Delta /6 \geq \rr(\UU)$, which is true whenever $r \geq 6$ (recall that we assume $\Delta \geq \rr(\UU)$).

The assertion that the barriers are entirely contained within $[c \Delta]^2$ if $c$ is sufficiently large follows immediately from the fact that by the choice of $\varepsilon$ every point of every $(i,t)$-barrier is at distance at most
\[
 a_{max} + r \Delta/2 + \rr(\UU) \leq \Delta \left (2r  \Bigg ( \tan \left ( \frac{\theta_3 - \theta_2}{2} \right ) +  \tan \left ( \frac{\theta_2 - \theta_1}{2} \right ) \Bigg ) +\frac{r}{2} + 1 \right )
\]
from $v$. Therefore, for $\varepsilon \leq \varepsilon_0$ and $r = \max \{ 6, r_\varepsilon \}$ the lemma holds with
\[
 \ell_0 = \left \lceil 2r  \Bigg ( \tan \left ( \frac{\theta_3 - \theta_2}{2} \right ) +  \tan \left ( \frac{\theta_2 - \theta_1}{2} \right ) \Bigg ) +\frac{r}{2} + 1 \right \rceil. \qedhere
\]
\end{proof}

Given a set of stable directions $\SS$ and our choice of strongly stable and not forbidden directions $u_1$, $u_2$ and $u_3$ in Section \ref{sec:chooseDirections}, let $c(\SS)$ be the smallest $c=2\ell+1$ for which Lemma \ref{lem:coverExists} holds for $\varepsilon = \min \{\varepsilon_0, \varepsilon(u_1, u_2, u_3)\}$.

Now let $K\subset\ZZ^2$ be finite and let $\Delta>0$ be minimal such that $K$ is contained in a square of side length $\Delta$. We say that an $(i)$-triangular cover for a finite set $K$ is \emph{tight} if it is completely contained in the $\big(c(\SS)\Delta\big)\times\big(c(\SS)\Delta\big)$ square centered at any minimal square (necessarily of side length $\Delta$) containing $K$.


\section{Positive critical probability: The proof of Theorem~\ref{thm:subpc}}
\label{sec:proof}

The aim of the first part of this section is to state and prove the theorem that will be our main tool in proving Theorem~\ref{thm:subpc}. We described the outline of the proof of this theorem in Section \ref{sec:outline}. Before stating the theorem, we need a few preliminary definitions and remarks.

For each $k\geq 1$, we say that the measure $\PP_p$ is \emph{$(k,2)$-independent} if, for every pair of non-adjacent squares $S$ and $T$ of side length $\Delta_k$ in the $(k)$-tiling, the events
\[
\{S \text{ is } (k) \text{-good} \} \quad \text{and} \quad \{T \text{ is } (k) \text{-good}\}
\]
are independent.

Recall that, for $i\geq 1$, a site $v$ in an $(i)$-good square is said to be clean if, for all $j<i$, $v$ is at distance at least $g_j/3$ from any $(j)$-bad square. 

In the statement of the theorem we refer to unions of pairwise adjacent $(i)$-bad $\Delta_i$-squares. Note that at most four such squares can be all pairwise adjacent and that a union of such squares is always contained in a $2\Delta_i \times 2\Delta_i$ square.

After the initial infection is seeded, bootstrap percolation is a fully deterministic process. Hence, given a set of initially infected sites $A \subset \ZZ^2$, for each $k\geq 1$ let $\XX_k$ be the collection of all sets $X \subset \ZZ^2$ such that $X$ is a union of pairwise adjacent $(k)$-bad squares and $X$ intersects a $(k+1)$-good square. 

\begin{theorem}
\label{thm:coversExist}
 Let $\UU$ be a subcritical update family with three strongly stable directions $u_1, u_2, u_3 \in \SSS \setminus F(\UU)$ such that for some positive numbers $\lambda_1, \lambda_2, \lambda_3$ we have $\lambda_1 u_1 + \lambda_2 u_2 + \lambda_3 u_3 = 0$. Then, if $p>0$ is small enough, for each $k \geq 1$ the following three conditions hold:
\begin{enumerate}[leftmargin=0.3in]
 \item \label{cond:cond1} The measure $\PP_p$ is $(k,2)$-independent, and for any $\Delta_k \times \Delta_k$ square $S$ in the $(k)$-tiling we have
\[
\PP_p( S \text{ is } (k)\text{-bad}) \leq q_k.
\]
 \item \label{cond:cond2} Every $(k)$-good square $S$ contains a $(k)$-clean site.
 \item \label{cond:cond3} For every $X \in \XX_k$ there exists a tight $(k)$-triangular cover $T_{k}(X)$ such that, for distinct $Y, Z \in \XX_k$, the sets $T_{k}(Y)$ and $T_{k}(Z)$ are disjoint, and for each $i<k$, if $Y \in \XX_{k}$ and $Z \in \XX_{i}$, then either $T_{k}(Y)$ and $T_{i}(Z)$ are disjoint or $T_{i}(Z) \subset T_{k}(Y)$.
\end{enumerate}
\end{theorem}

Since all $(k)$-triangular covers we consider henceforth will be tight, we shall always assume that this extra condition is understood, and make no further mention of it.

\begin{proof}
Given a choice of $\alpha, \beta, \gamma$ and $\delta$ satisfying \eqref{eq:parameters}, and $\varepsilon = \min \{\varepsilon_0, \varepsilon(u_1, u_2, u_3)\}$ where $\varepsilon_0$ is taken as in the proof of Lemma \ref{lem:coverExists} and $\varepsilon(u_1, u_2, u_3)$ as in Section \ref{sec:chooseDirections}, let $\Delta_1$ be large enough to satisfy the following five conditions:
\begin{itemize}
 \item $\Delta_1 \geq \max \{ 2^{\delta+5}, \rr(\UU) \},$
 \item $\Delta_1^{\alpha-1} \geq 12 c(\SS),$
 \item $\Delta_1^{\beta-1} \geq \max \{ 30, 3 c(\SS) \},$
 \item $\Delta_1^{\alpha-\beta} \geq 3,$
 \item $\Delta_1^{\beta-1-\gamma} \geq 68 c(\SS) / \varepsilon.$
\end{itemize}
Let sites in $\ZZ^2$ be initially infected independently with probability $p = (\Delta_1)^{-\delta-2}$. We shall prove Theorem \ref{thm:coversExist} by induction on $k \geq 1$. First we check the case $k=1$.
\begin{enumerate}[leftmargin=0.3in]
 \item Any $\Delta_1$-square $A$ is $(1)$-good if it is initially fully healthy. Thus we immediately see that states of all $\Delta_1$-squares are mutually independent. We also have
\[
 \PP_p(A \text{ is } (1) \text{-bad}) < (\Delta_1)^2 p = (\Delta_1)^2 (\Delta_1)^{-\delta-2} = (\Delta_1)^{-\delta} = q_1.
\]
 \item Every site in a $(1)$-good $\Delta_1$-square is $(1)$-clean (the condition of a $(1)$-clean site is empty) and therefore Condition (ii) is trivially satisfied by any $(1)$-good $\Delta_1$-square.
 \item For $k=1$ Condition (iii) is empty and is therefore trivially satisfied by any $(1)$-good $\Delta_1$-square.
\end{enumerate}

Assume now that the three conditions of Theorem \ref{thm:coversExist} are satisfied by our $(i)$-tilings for all $1 \leq i \leq k$. Let us consider the $(k+1)$-tiling of $\ZZ^2$.
\begin{enumerate}[leftmargin=0.3in]
 \item The state of any square $X$ in our $(k+1)$-tiling (either ``$(k+1)$-good'' or ``$(k+1)$-bad'') depends only on the states of squares in the $(k)$-tiling within distance $g_k = \Delta_k^\beta$ of $X$. If $\Delta_1^{\alpha - \beta} \geq 3$ then for all $k$ we have $g_k \leq \Delta_{k+1}/3$ and the states of any non-adjacent $\Delta_{k+1}$-squares $Y$ and $Z$ depend on states of non-adjacent sets of $\Delta_{k}$-squares. By induction, the states of squares in these non-adjacent sets are independent. Therefore the states of $Y$ and $Z$ are independent. Hence the states of all non-adjacent $\Delta_{k+1}$-squares are independent.

If a $\Delta_{k+1} \times \Delta_{k+1}$ square $S$ is $(k+1)$-bad then it contains or is at distance at most $g_k$ from two non-adjacent $(k)$-bad squares $X, Y$ in our $(k)$-tiling such that $\dist(X,Y) \leq g_k$. Hence, given $S$, there are at most
\[
 \left ( \frac{\Delta_{k+1}+\Delta_k+2g_k}{\Delta_k} \right )^2
\]
ways of choosing $X$ and then, assuming that $Y$ is contained in the semicircle of radius $g_k$ below $X$, we have $2 (g_k/\Delta_k)^2$ ways of choosing $Y$.  Recall that $\Delta_{k+1} < \Delta_{k}^{\alpha} + \Delta_{k}$ and that for all $k \geq 1$ we have $\Delta_k \geq \Delta_1 \geq 2^{\delta+5}$. Since $q_k = \Delta_k^{-\delta}$, where $\delta = (2\alpha+2\beta-3)/(2-\alpha)$, and the states of non-adjacent squares are independent, we have
\[
\begin{split}
 \PP_p(A \text{ is } (k+1)\text{-bad}) & < \left ( \frac{\Delta_{k+1}+\Delta_k+2g_k}{\Delta_k} \right )^2 2 \left ( \frac{g_k}{\Delta_k} \right )^2 q_k^2 \\
	& < 2 \left ( \frac{4 \Delta_{k+1}}{\Delta_k} \right )^2 \left ( \Delta_k^{\beta-1} \right )^2 \Delta_k^{-2 \delta} \\
	& < 2 \left ( 4 \Delta_{k}^{\alpha-1} \right )^2 \Delta_k^{2\beta-2 -2 \delta} \\
	& = 2^{5+\delta} \Delta_k^{-1} 2^{-\delta} \Delta_k^{2 \alpha + 2\beta-3 -2 \delta}\\
	& \leq  2^{-\delta} \Delta_k^{2\alpha+2 \beta -3 - 2 \delta} \\
	& = (2 \Delta_{k}^{\alpha})^{- \delta} \\
	& \leq q_{k+1}.
\end{split}
\]

\item If a $(k+1)$-good square $S$ does not contain any $(k)$-bad subsquare then, in particular, any square $Y$ in our $(k)$-tiling contained in the middle $\Delta_{k+1}/3 \times \Delta_{k+1}/3$ subsquare of $S$ is $(k)$-good and lies at distance at least $\Delta_{k+1}/3 > g_k/3$ from any $(k)$-bad square. Since $Y$ is $(k)$-good it contains a $(k)$-clean site $v$. Since $v$ is at distance at least $g_k/3$ from any $(k)$-bad square, $v$ is also $(k+1)$-clean.

Hence assume that $S$ contains a $(k)$-bad square $X$. Since $S$ is $(k+1)$-good, any other $(k)$-bad square within distance $g_k$ of $X$ (not necessarily contained in $S$) must be adjacent to $X$. It follows that, since $\Delta_1^{\beta-1} \geq 30$, every site at distance between $2g_k/5$ and $3g_k/5$ from $X$ is at distance at least $g_k/3$ from any $(k)$-bad square. At least a quarter of the ring of sites at distance between $2g_k/5$ and $3g_k/5$ from $X$ lies inside $S$. Additionally, this ring is thick enough to contain a $3 \Delta_k \times 3 \Delta_k$ square, which itself contains a $(k)$-good square with a $(k)$-clean site $v$. By the same argument as in the previous paragraph, $v$ is also $(k+1)$-clean.

\item Consider a $(k+1)$-good square $S$ and a union $X$ of pairwise adjacent $(k)$-bad squares intersecting $S$ (as usual, $X$ is contained within a $2 \Delta_k \times 2\Delta_k$ square). By Lemma \ref{lem:coverExists}, the definition of a $(k+1)$-good square, and since $\Delta_{1}^{\beta-1} \geq 3c(\SS)$, the $2c(\SS)\Delta_k \times 2c(\SS)\Delta_k$ square $C$ centered at $X$ does not intersect with the $2c(\SS)\Delta_k \times 2c(\SS)\Delta_k$ square centered at any other union of adjacent $(k)$-bad squares. Additionally, $C$ contains three $(k)$-good squares $C_1, C_2$ and $C_3$ with $(k)$-clean sites $c_1 \in C_1$, $c_2 \in C_2$ and $c_3 \in C_3$ such that all $(k,1)$-barriers joining $c_1$ to $c_2$, all $(k,2)$-barriers joining $c_2$ to $c_3$ and all $(k,3)$-barriers joining $c_3$ to $c_1$ are contained within $C$. Also, these barriers are disjoint from $X$, which lies inside the area bounded by them.

Therefore we need to prove that between any two of $c_1, c_2$ and $c_3$ we can find appropriate barriers avoiding $T_i(Y)$ for any union $Y$ of adjacent $(i)$-bad squares for all $i < k$. Then the union of these three barriers and the area inside them will be our desired $T_{k}(X)$, the $(k)$-triangular cover of $X$. To do this we shall prove the following crucial lemma. We would like to emphasize that this lemma is the key to the third and most important part of Theorem~\ref{thm:coversExist}. The theorem follows from the lemma in an essentially straightforward way.

\begin{lemma}
\label{lem:barriersExist}
 Let $j \geq 1$. Let $x_0$ and $y_0$ be two $(j)$-clean sites in different $(j)$-good squares such that for some $t \in \{1,2,3\}$ we have
\[
 \left | \big (\theta(u_{x_0,y_0}) - \theta(u_t) \big ) (\text{mod } 2\pi) - \pi/2 \right | < \sigma_{j} = \varepsilon/2 + \varepsilon/\Delta_j^\gamma.
\]
Suppose also that all $\Delta_j\times\Delta_j$ squares in our $(j)$-tiling within distance $\Delta_{j}$ of the segment with $x_0$ and $y_0$ as endpoints are $(j)$-good. Then there exists a $(j,t)$-barrier joining $x_0$ to $y_0$ that does not intersect the $(i)$-triangular cover $T_i(X)$ of any union $X$ of neighbouring $(i)$-bad squares for any $i<j$.
\end{lemma}

\begin{proofOfLemma}
For $j=1$ the assertion is empty and so the lemma is trivial. Thus assume that the lemma holds for $j \leq m$. Let $x_0$ and $y_0$ be two $(m+1)$-clean sites in different $(m+1)$-good squares such that, for some $t \in \{1,2,3\}$,
\[
 \left | \big (\theta(u_{x_0,y_0}) - \theta(u_t) \big ) (\text{mod } 2\pi) - \pi/2 \right | < \sigma_{m+1}
\]
holds. Recall that every $(m+1)$-clean site is also $(m)$-clean.

Let
\[
\begin{split}
 x_1 & = x_0 + 8 c(\SS) \Delta_{m} u \big (\theta(u_{x,y})+\pi/2 \big ), \\
 y_1 & = y_0 + 8 c(\SS) \Delta_{m} u \big (\theta(u_{x,y})+\pi/2 \big ), \\
 x_2 & = x_0 + 8 c(\SS) \Delta_{m} u \big (\theta(u_{x,y})-\pi/2 \big ), \\
 y_2 & = y_0 + 8 c(\SS) \Delta_{m} u \big (\theta(u_{x,y})-\pi/2 \big ),
\end{split}
\]
and for $\ell = 0,1,2$ let
\[
 \qquad Z_\ell = \{v \in \ZZ^2 : \dist(v, \lambda x_\ell + (1-\lambda) y_\ell) \leq 4 c(\SS) \Delta_{m}  \text{ for some } \lambda \in [0,1]\}
\]
(see Figure~\ref{fig:detour}).

If $\Delta_{m+1} \geq 12 c(\SS) \Delta_m$, which is true since $\Delta_{1}^{\alpha-1} \geq 12 c(\SS)$, then $\bigcup_{\ell = 1}^3 Z_\ell$ is contained in a union of $(m+1)$-good squares. This implies that every union of pairwise adjacent $(m)$-bad squares intersecting $\bigcup_{\ell = 1}^3 Z_\ell$ is at distance at least $g_m$ from any other $(m)$-bad square. Additionally, the $(m)$-triangular cover of any union $X$ of pairwise adjacent $(m)$-bad squares, being contained in the $2 c(\SS) \Delta_m \times 2 c(\SS) \Delta_m$ square centered at $X$, intersects at most two of the sets $Z_\ell$.

Assume that $\bigcup_{\ell = 1}^3 Z_\ell$ intersects $d$ such $2 c(\SS) \Delta_m \times 2 c(\SS) \Delta_m$ squares containing unions of adjacent $(m)$-bad squares: $Y_1, Y_2, \ldots , Y_d$, ordered according to their distance from $x_0$. For every $s \in [d]$, let $y_s \in \RR^2$ be the centre of $Y_s$ and let $\ell_s \in \{1,2\}$ be an index of a set $Z_\ell$ that is avoided by $Y_s$. Then in $Z_{\ell_s}$ we can find an $(m)$-good square $C_{s}$ at distance at least $4c(\SS) \Delta_m$ and at most $6c(\SS) \Delta_m$ from $Z_0$, with an $(m)$-clean site $z_s \in C_s$, such that the distance between $z_{s}$ and the line going through $x_0, x_1$ and $x_2$ differs from the distance between $y_s$ and that line by at most $\Delta_m$. Note that the conditions on the location of $C_{s}$ imply that $C_{s}$ is at distance at least $3 c(\SS) \Delta_m/2$ from $\ZZ^2 \setminus Z_{\ell_s}$. See Figure \ref{fig:detour} for a graphical interpretation of this description.

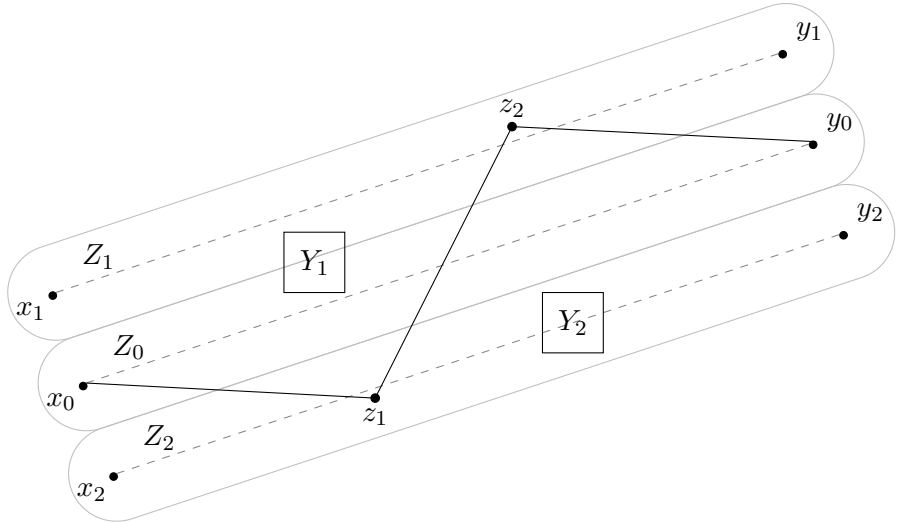
\begin{figure}[ht]
  \centering
  \begin{tikzpicture}[scale=\latticescale]

    \foreach \xi/\yi/\xii/\yii/\j in {0/0/12/4/0, -0.5/1.5/11.5/5.5/1, 0.5/-1.5/12.5/2.5/2} {
      \draw [below left, help lines, dashed] (\xi,\yi) -- (\xii,\yii);
      \draw (\xi,\yi) [below left] node {$x_\j$};
      \draw (\xii,\yii) [above right] node {$y_\j$};
      \node[draw,circle,inner sep=0,minimum size=0.1cm,fill,below left] at (\xi,\yi) {};
      \node[draw,circle,inner sep=0,minimum size=0.1cm,fill,below left] at (\xii,\yii) {};
      \draw [lightgray, xshift=\xi cm, yshift=\yi cm] (108.435:0.7906) arc (108.435:288.435:0.7906) -- ++(12,4) arc (-71.565:108.435:0.7906) -- cycle;
      \draw ({\xi+0.7},{\yi+0.6}) node {$Z_\j$};
    }

    \draw (0,0) -- plot[mark=*] coordinates {(4.75,-0.25)} -- plot[mark=*] coordinates {(7,4.25)} -- (12,4);
    \draw (3.25,1.5) rectangle +(1,1);
    \draw (7.5,0.5) rectangle +(1,1);
    \draw (3.755,2) node {$Y_1$};
    \draw (8,1) node {$Y_2$};
    \draw (4.75,-0.25) [below] node {$z_1$};
    \draw (7,4.25) [above] node {$z_2$};

  \end{tikzpicture}
  \caption{The $(m)$-clean sites $z_1$ and $z_2$ used to bypass unions of adjacent $(m)$-bad squares $Y_1$ and $Y_2$ and to inductively construct an $(m+1,t)$-barrier joining $x_0$ to $y_0$.}
  \label{fig:detour}
\end{figure}

Set $z_0 = x_0$ and $z_{d+1} = y_0$. Note that if the segment joining $z_s$ to $z_{s+1}$ is at distance at least $\Delta_m$ from any $(m)$-triangular cover of any union of adjacent $(m)$-bad squares (this clearly implies that the segment is at distance at least $\Delta_m$ from any $(m)$-bad square) and if
\[
 \left | \big (\theta(u_{z_s,z_{s+1}}) - \theta(u_t) \big ) (\text{mod } 2\pi) - \pi/2 \right | < \sigma_{m},
\]
then by the induction hypothesis there exists an $(m,t)$-barrier joining $z_s$ to $z_{s+1}$ satisfying the lemma. If this holds for all pairs of consecutive $z_s$s then these $(m,t)$-barriers together constitute an $(m+1,t)$-barrier joining $x_0$ to $y_0$ which avoids, for all $i \leq m$, $(i)$-triangular covers of all unions of neighbouring $(i)$-bad squares.

Since $\Delta_{1}^{\beta-1} \geq 30$, using the bound $\arcsin \phi \leq \pi \phi / 2$ for $\phi \in [0,1]$, the difference between $\theta(u_{z_s,z_{s+1}})$ and $\theta(u_{x_0,y_0})$ modulo $2 \pi$ is bounded from above by
\[
 \arcsin{\left ( \frac{20 c(\SS) \Delta_{m}}{g_m-2\Delta_{m}} \right ) } < \frac{75 \pi c(\SS) \Delta_{m}}{7 g_m} < \frac{\varepsilon}{2\Delta_m^\gamma}
\]
for all $m \geq 1$ since $\Delta_{1}^{\beta-1-\gamma} \geq 68 c(\SS)/\varepsilon$. Since
\[
 \sigma_{m} - \sigma_{m+1} = \frac{\varepsilon}{\Delta_m^{\gamma}} - \frac{\varepsilon}{\Delta_{m+1}^{\gamma}} > \frac{\varepsilon}{2\Delta_m^{\gamma}},
\]
we see that for $\Delta_{1} \geq (68 c(\SS)/\varepsilon)^{1/(\beta-1-\gamma)} $ the angles between consecutive $z_s$s allow us to find $(m,t)$-barriers between these sites.

Let us then show that the segment joining $z_s$ to $z_{s+1}$ is at distance at least $\Delta_m$ from any $(m)$-triangular cover of any union of adjacent $(m)$-bad squares. First, we observe that $z_s$ and $z_{s+1}$ are at distance at least $3 c(\SS) \Delta_m/2$ from $\ZZ^2 \setminus \bigcup_{\ell = 1}^3 Z_\ell$, so we do not need to consider $(m)$-bad squares lying outside $\bigcup_{\ell = 1}^3 Z_\ell$.

We chose $z_s$ to be at distance at least $4 c(\SS) \Delta_m$ from $Z_0$, and consequently also from $Y_s$. Let $w'$ be a site in a $(m)$-triangular cover of $Y_s$. Then, by Lemma \ref{lem:coverExists}, the distance between $w$ and the line going through $x_0, x_1$ and $x_2$ is not larger than the distance between $z_s$ and this line by more than $2 c(\SS) \Delta_m$. Let $w''$ be a point in the segment joining $z_s$ to $z_{s+1}$ at distance at most $2c(\SS) \Delta_m$ from $Z_0$. If $\varepsilon/(2\Delta_m^\gamma) \leq \pi/8$, which is true whenever $\varepsilon \leq \pi/4$, then $w''$ is at distance from the line going through $x_0, x_1$ and $x_2$ larger by at least $4 c(\SS) \Delta_m$ than $z_s$ is. Therefore, the segment joining $z_s$ to $z_{s+1}$ is at distance at least $2c(\SS) \Delta_m$ from $Y_s$ and so at distance at least $\Delta_m$ from $T_m(Y_s)$. In a similar way we show that it is at distance at least $\Delta_m$ from $T_m(Y_{s+1})$. By the choice of the ordering of the squares $Y_s$ we know that no other $(m)$-triangular cover of any union of adjacent $(m)$-bad squares is near the segment joining $z_s$ to $z_{s+1}$ and the lemma is proved.
\end{proofOfLemma}

From Lemma \ref{lem:coverExists} and Lemma \ref{lem:barriersExist} it follows immediately that for any union $X$ of adjacent $(k)$-bad squares inside a $(k+1)$-good square we can find a $(k)$-triangular cover $T_k(X)$ of $X$ inside the $2 c(\SS) \Delta_k \times 2 c(\SS) \Delta_k$ square centered at $X$, satisfying the requirements of Theorem \ref{thm:coversExist}.
\end{enumerate}
This completes the proof of the theorem.
\end{proof}

In the next lemma we show that the collection of triangular covers, which by Theorem \ref{thm:coversExist} almost surely exists if $p > 0$ is sufficiently small, contains every site of $\ZZ^2$ that ever becomes infected.

Recall that $\XX_k$ is the collection of all sets $X\subset\ZZ^2$ such that $X$ is a union of pairwise adjacent $(k)$-bad squares and $X$ intersects a $(k+1)$-good square.

\begin{lemma}
\label{lem:mainFollows}
Given a subcritical family $\UU$, let $p = (\Delta_1)^{-\delta-2} > 0$ be small enough so that Theorem \ref{thm:coversExist} holds. Let $A \sim \Bin(\ZZ^2,p)$. Then, almost surely,
\[
 \closure{A} \subset Z = \bigcup_{i \geq 1} \bigcup_{X\in\XX_i} T_i(X).
\]
\end{lemma}
\begin{proof}
By the definition of the closure, the set $\closure{A}$ is the smallest set that contains $A$ and is closed under $\UU$.

We show first that $A\subset Z$. Note that since we define $q_i = \Delta_i^{-\delta} \leq \Delta_1^{-\alpha^{i-1} \delta}$, we have $\sum_{i \geq 1} q_i < \infty$. Every $\Delta_1$-square that contains at least one initially infected site is $(1)$-bad and, by the Borel-Cantelli lemma, $\sum_{i \geq 1} q_i < \infty$ implies that every site in $\ZZ^2$ is contained in infinitely many good squares almost surely. In particular, every initially infected site will be contained in the triangular cover of a union of adjacent $(i)$-bad squares intersecting an $(i+1)$-good square, for some $i\geq 1$. Thus to prove the lemma we just need to show that $Z$ is closed under $\UU$.

As shown in Lemma \ref{lem:bordersSurvive}, for any $i \geq 1$ the $(i)$-triangular cover of any union $X$ of adjacent $(i)$-bad squares is closed under $\UU$. Moreover, the infected interior of the cover is separated from $\ZZ^2 \setminus T_i(X)$ by a healthy barrier of thickness at least $\rr(\UU)$. By condition (iii) in Theorem \ref{thm:coversExist}, for all $i \geq j \geq 1$, any union $X$ of adjacent $(i)$-bad squares and any union $Y$ of adjacent $(j)$-bad squares satisfy either $T_j(Y) \subset T_i(X)$ or $T_j(Y) \cap T_i(X) = \emptyset$. Hence, by the definition of $\rr(\UU)$, any collection of triangular covers is closed under $\UU$ and, in particular, so is $Z$. This means that $\closure{A} \subset Z$ and the proof of the lemma is complete.
\end{proof}

Equipped with Theorem \ref{thm:coversExist} and Lemma \ref{lem:mainFollows}, we are now in a position to prove Theorem \ref{thm:subpc}.

\begin{proofOfMain}
Having proved Theorem \ref{thm:coversExist} and Lemma \ref{lem:mainFollows}, to prove the inequality $p_c(\ZZ^2,\UU) > 0$ in Theorem \ref{thm:subpc} it is enough to show that for $p > 0$ small enough the probability that there exists $i \geq 1$ and a union $X$ of adjacent $(i)$-bad squares such that the site $(0,0)$ belongs to the $2c(\SS) \Delta_i \times 2c(\SS) \Delta_i$ square centered at $X$ is strictly less than $1$. This clearly implies that the probability that the origin belongs to some $(i)$-triangular cover of adjacent $(i)$-bad squares is strictly less than $1$.

Given $\alpha$, $\beta$, $\gamma$ and $\delta$ satisfying \eqref{eq:parameters}, let $\Delta_1$ be large enough to satisfy all conditions imposed on it at the beginning of the proof of Theorem \ref{thm:coversExist}. Since in the proof of Theorem \ref{thm:coversExist} we take $p = (\Delta_1)^{-\delta-2}$, this implies an appropriate condition on $p$.

The probability that there exists $i \geq 1$ and a union $X$ of adjacent $(i)$-bad squares such that the site $(0,0)$ belongs to the $2c(\SS) \Delta_i \times 2c(\SS) \Delta_i$ square centered at $X$ can be bounded from above by the expected number of such squares, which is at most
\[
\sum_{i \geq 1} (2c(\SS)+2)^2 q_i \leq 5(c(\SS))^2 \sum_{i \geq 1} \Delta_i^{-\delta} \leq 5(c(\SS))^2 \sum_{i \geq 0} \Delta_1^{-\delta \alpha^{i}}.
\]
We have $\delta = \frac{2\alpha+2\beta-3}{2-\alpha} > 1$ and so, in the proof of Theorem \ref{thm:coversExist},
\[
p = (\Delta_1)^{-\delta-2} > \Delta_1^{-3 \delta}.
\]
Therefore we obtain
\[
\begin{split}
\PP_p ([A] = \ZZ^2) & \leq 5(c(\SS))^2 \sum_{i \geq 0} p^{\alpha^{i}/3} \\
		    & \leq 5(c(\SS))^2 \left ( p^{1/3} + \sum_{i \geq 1} p^{(\alpha \log \alpha i + \alpha ( 1 - \log \alpha ))/3} \right ),
\end{split}
\]
where in the second inequality we use the convexity of the function $f(x) = \alpha^x$, which implies $f(x) \geq f(1)+f'(1)(x-1)$. With $p < 2^{-3/(\alpha \log \alpha)}$ it follows that
\begin{equation}
\label{eq:finalBound}
\PP_p ([A] = \ZZ^2) \leq 5(c(\SS))^2 \left ( p^{1/3} + 2 p^{\alpha/3} \right ).
\end{equation}
Thus if $5(c(\SS))^2 \left ( p^{1/3} + 2 p^{\alpha/3} \right ) < 1$ then $p \leq p_c(\ZZ^2,\UU)$ and the proof of the inequality $p_c(\ZZ^2,\UU) > 0$ in Theorem \ref{thm:subpc} is complete.

We finally prove that $p_c(\ZZ^2,\UU) = 1$ if and only if $\SS=S^1$. To show that $\SS \neq S^1$ implies $p_c(\ZZ^2,\UU) < 1$ we couple bootstrap percolation with site percolation, using a standard argument. If we initially infect all sites in $\ZZ^2$ independently with probability $p < 1$ large enough then almost surely every initially healthy cluster in $\ZZ^2$ is not only finite, but is also surrounded by an annulus of initially infected sites of thickness at least $\rr(\UU)$. Then, if $u \in S^1 \setminus \SS$, we must have an $X_i \in \UU$ such that $X_i \subset \HH_u$ and every finite cluster of healthy sites is infected by the dynamics with the use of update rule $X_i$.

To show the converse we use following simple argument. Assume that $\SS=S^1$, so that all update rules in $\UU$ do not destabilize any direction, i.e., for all $i \in [m]$ the origin belongs to the convex hull of $X_i$. For any $r > 0$ and $p < 1$, if we initially infect all sites in $\ZZ^2$ with probability $p$ then almost surely somewhere in $\ZZ^2$ we obtain an initially healthy disk $D_r$ of radius $r$. If $r$ is large enough then every rule $X_i$ can only infect sites in disjoint circular segments ``cut off'' from $D_r$ using chords of length at most $\rr(\UU)$ and parallel to the sides of the convex hull of $X_i$, and these segments are all either disjoint or contained in each other for different rules (that again follows from the fact that we take $r$ large, see Figure \ref{fig:disjointChords}). Because no additional infection takes place in $D_r$, we do not have percolation. That completes the proof of Theorem \ref{thm:subpc}.

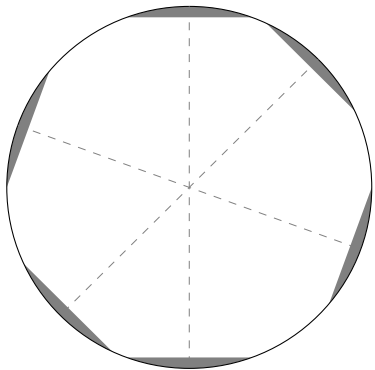
\begin{figure}[ht]
  \centering
  \begin{tikzpicture}[scale=\tikzscale]
    \fill [gray] (25:2) arc (25:65:2) -- cycle;
    \fill [gray] (205:2) arc (205:245:2) -- cycle;
    \fill [gray] (70:2) arc (70:110:2) -- cycle;
    \fill [gray] (250:2) arc (250:290:2) -- cycle;
    \fill [gray] (140:2) arc (140:180:2) -- cycle;
    \fill [gray] (320:2) arc (320:360:2) -- cycle;
    \draw [black] (0,0) circle (2);
    \draw [help lines,dashed,gray] (45:2) -- (225:2);
    \draw [help lines,dashed,gray] (90:2) -- (270:2);
    \draw [help lines,dashed,gray] (160:2) -- (340:2);
  \end{tikzpicture}
  \caption{Set of disjoint circular segments cut off from $D_r$ using chords perpendicular to directions $u(\theta)$ for $\theta \in \{\pi/4, \pi/2, 8\pi/9 \}$.}
  \label{fig:disjointChords}
\end{figure}
\end{proofOfMain}

We finally prove the lower bound on $p_c(\triD,2)$ in Corollary \ref{cor:DTBP}. We emphasize that because our proof is very general, the bounds it gives in specific cases are likely to be far from optimal.
\begin{proofOfDTBP}
For the update family $\UU_1$ equivalent to DTBP we have $\rr(\UU_1) = \dist((-1,-1),(0,1)) = \sqrt{5} < 2.24$. Since in \eqref{eq:vecsum} we are free to take any $u_1$, $u_2$ and $u_3$ that satisfy this equation for some positive values of the $\lambda_i$ and lie inside open intervals of stable directions that do not intersect the forbidden set, we choose $\theta(u_1) = 7\pi/24$, $\theta(u_2) = 23\pi/24$ and $\theta(u_3) = 39\pi/24$. This implies that $\theta(u_3) - \theta(u_2) = \theta(u_2) - \theta(u_1) = \theta(u_1) - \theta(u_3) = 2 \pi / 3$. Also, for $t=1,2,3$ and $|\theta(u) - \theta(u_t)| < \pi/24$, direction $u$ is stable and not forbidden.

From these values of $\theta(u_t)$ we get $\varepsilon_0 > 0.02293 \pi$ and $r_{\varepsilon_0} < 24.04$. This gives $c(\SS) = 361$. We choose $\alpha = 1.5$ and simplifying in \eqref{eq:finalBound} we obtain $\PP_p ([A] = \ZZ^2) \leq 15(c(\SS))^2 p^{1/3}$, which is less than $1$ when $p < 10^{-19}$. This implies the condition $\Delta_1 > 10^{19/(\delta+2)}$. Taking $\beta = 1.45$ and $\gamma = 0.01$, this condition and the ones at the beginning of the proof of Theorem \ref{thm:coversExist} are satisfied for $\Delta_1 \geq 10^{13}$. Since we have $\delta = 5.8$ this implies that $p_c(\triD,2) > 2.5 \cdot 10^{-101}$ and the proof of Corollary \ref{cor:DTBP} is complete.
\end{proofOfDTBP}

\section{Update families with two opposite strongly stable directions}
\label{sec:percolationCoupling}

In this section we present an elementary proof of the fact that the critical probability is strictly positive for all update families with two opposite strongly stable directions, i.e., for families $\UU$ such that for some $u \in S^1$ we have $u, -u \in \SSS(\UU)$. The following theorem is of course only a particular subcase of Theorem \ref{thm:subpc} but it covers all previously analysed subcritical bootstrap percolation models \cite{criticalPoints,glassTransition,jammingPercolation,forceBalance}. (Of course, the point of those papers was not, as here, to prove that the critical probability is positive, but rather to determine quite precise information about its location.)

\begin{theorem}
\label{thm:percolationCoupling}
 For every update family $\UU$ such that $\{u, -u\} \subset \SSS(\UU)$ for some $u \in S^1$, we have $p_c(\ZZ^2,\UU) > 0$.
\end{theorem}
\begin{proof}
 Choose $u' \in \SS$ and $\varepsilon > 0$ such that $N_{\varepsilon}(u'), N_{\varepsilon}(-u') \subset \SSS \setminus F(\UU)$. Tile $\ZZ^2$ with identical rhombi, whose sides are perpendicular to the four directions $u(\theta(\pm u') \pm \varepsilon/2)$, and which are large enough to contain a circle of radius $r \geq \rr(\UU)$. If $p > 0$ is small enough then every rhombus is initially fully healthy with probability larger than the critical probability for oriented site percolation, independently of all other rhombi. Hence in the tiling we almost surely have an infinite ``increasing'' path of fully healthy rhombi which, by the choice of $u'$, $\varepsilon$ and $r$, remains healthy forever.
\end{proof}

\section{Open problems}
\label{sec:openProblems}

When $p>p_c$, the sorts of questions one typically asks of critical bootstrap and $\UU$-bootstrap percolation become relevant to subcritical $\UU$-bootstrap percolation. For example, one would like to know about the distribution of the occupation time $T$ of the origin, and in particular, to what extent this time is concentrated, and how its expectation behaves as $p \searrow p_c$. These questions have been extensively studied in the case of the $r$-neighbour model on $\ZZ^d$ and are the subject of a number of recent results for critical update families in $\UU$-bootstrap percolation. It is natural to ask whether similar behaviour occurs in the subcritical setting. Some of the following questions (e.g., Question \ref{qu:scaling} and \ref{qu:critBehaviour}) have already been addressed in \cite{glassTransition} for models that can be coupled with oriented site percolation. However, the methods used in \cite{glassTransition} strongly depend on the coupling idea and cannot be applied to ``typical'' subcritical update families. It is therefore unclear whether the models with no two opposite strongly stable directions share similar behaviour.

\begin{question}
\label{qu:scaling}
{\bf (Scaling limit of $T$.)} What is the behaviour of\;$T$ as $p \searrow p_c$? In particular, does $T$ tend to infinity, and if so, what is the limiting dependence of $T$ on $p-p_c$?
\end{question}

The non-triviality of the critical probabilities of subcritical $\UU$-bootstrap percolation models also opens up the area to the sorts of questions one typically asks of traditional Bernoulli (site or bond) percolation. The difficulty of answering these questions is likely to be correlated with the difficulty of answering the corresponding questions in Bernoulli percolation: for example, determining the exact value of $p_c$, or even obtaining good bounds on $p_c$, for any non-trivial subcritical update family, is likely to be a hard problem. Similarly, properties conjectured to have critical exponent behaviour in Bernoulli percolation, such as the distribution of cluster sizes, are likely to be hard to analyse in the subcritical $\UU$-bootstrap percolation setting. However, there are many properties of site and bond percolation that are now well-understood, at least in two dimensions, and these may also be accessible in the subcritical $\UU$-bootstrap percolation setting. We give three examples: the behaviour at criticality, exponential decay of cluster sizes, and noise sensitivity.

\begin{question}
\label{qu:critBehaviour}
{\bf (Behaviour at criticality.)} Is there percolation almost surely when $p=p_c$? If so, do we have $\EE T<\infty$?
\end{question}

Let $\PP_p(0\leftrightarrow r)$ denote the probability that the origin is contained in a connected component of radius at least $r$ (according to an arbitrary norm) in the closure of $A$.

\begin{question}
\label{qu:exp}
{\bf (Exponential decay.)} For $p<p_c$, does $\PP_p(0\leftrightarrow r)$ decay exponentially in $r$?
\end{question}

Here we mean `connected' in the site percolation sense, although other notions of connectedness are also interesting. It is not clear that one should
expect a positive answer to Question \ref{qu:exp}: the droplet-like geometry of the closure of a random initial set suggests that perhaps the distribution may be much flatter.

In the context of random discrete structures, roughly speaking \emph{noise sensitivity} measures whether small perturbations of a system asymptotically cause all information to be lost. The theory of noise sensitivity was introduced by Benjamini, Kalai and Schramm \cite{BKSnoise}, who were motivated by applications to exceptional times in dynamical percolation, and it was later developed by Garban, Pete, and Schramm \cite{GPSnoise}, and by Schramm and Steif \cite{SSnoise}. Rather than giving the precise definitions we refer the reader to the articles above for an overview, and we mention that in the subcritical $\UU$-bootstrap percolation setting one can define a corresponding notion.

\begin{question}
{\bf (Noise sensitivity.)} Are subcritical $\UU$-bootstrap percolation models noise sensitive at $p=p_c$?
\end{question}

We end with a number of questions of a different flavour, which cannot be asked of critical $\UU$-bootstrap percolation or of Bernoulli percolation, but which are interesting in their own right. First, let $C^\infty$ denote the event that there exists an infinite connected component in the closure of $A$. Observe that $C^\infty$ is translation invariant, so by ergodicity it has probability either $0$ or $1$. Combining this with monotonicity, it follows that there is a critical probability $p_c^\infty=p_c^\infty(\UU)$ such that
\[
\PP_p(C^\infty) = \begin{cases} 0 \text{ if } p<p_c^\infty \\ 1 \text{
if } p>p_c^\infty. \end{cases}
\]
It is natural ask about the relationship between $p_c$ and $p_c^\infty$: trivially the inequality $p_c^\infty\leq p_c$ always holds, but is it possible to have strict inequality? Even if not, could it be that $\PP_{p_c}(C^\infty)=1$ but $\PP_{p_c}([A]=\ZZ^2)=0$?

\begin{question}
{\bf (Infinite component without percolation.)} For which subcritical $\UU$-bootstrap percolation models do we have $p_c^\infty=p_c$?
\end{question}

This question does not seem to have been studied even in the case of oriented site percolation.

Define the random variable
\[
D(n) = \frac{\big|[-n,n]^2\cap[A]\big|}{\big|[-n,n]^2\big|}.
\]
Thus, $D(n)$ is the density of the closure $[A]$ inside the square $[-n,n]^2$. Analogous to numerous phenomena, we conjecture the following.

\begin{conjecture}\label{co:density}
{\bf (Density of the closure.)} For every $p\in[0,1]$ there exists a constant $\delta(p)$ such that $D(n)$ converge in probability to a constant $\delta(p)$ as $n\to\infty$.
\end{conjecture}

This conjecture is one formulation of the assertion that sites in the closure of $A$ should be reasonably well scattered. If Conjecture \ref{co:density} is true, one would like to know if $\delta(p)$ is continuous at $p=p_c$, and whether we have $\delta(p)-p=o(p)$ as $p\to 0$.



At present, essentially nothing is known about $\UU$-bootstrap percolation in higher dimensions. Let $d\geq 2$ be an integer and let $\UU$ be a $d$-dimensional update family. We define the \emph{stable set} in $d$ dimensions completely analogously to in 2 dimensions. First, given $(d-1)$-sphere $S^{d-1} \subset \RR^d$, for each $u \in S^{d-1}$, let $\HH_u^d:=\{x\in\ZZ^d:\langle x,u \rangle<0\}$ be a half-space normal to $u$. Then the stable set is
\[
\SS = \SS(\UU) = \big\{ u\in S^{d-1} : [\HH_u^d]=\HH_u^d \big\}.
\]

Let $\mu : \mathcal{L}(S^{d-1})\to\RR$ be the Lebesgue measure on the collection of Lebesgue-measurable subsets of $S^{d-1}$. We define the $d$-dimensional family $\UU$ to be \emph{subcritical} if $\mu(H\cap\SS)>0$ for every hemisphere $H\subset S^{d-1}$. Note that this corresponds to the definition given at the start of the paper in the special case $d=2$. We conjecture the following.

\begin{conjecture}\label{con:highd}
Fix an integer $d\geq 2$ and let $\UU$ be a $d$-dimensional update family. Then $p_c(\ZZ^d,\UU)>0$ if and only if $\UU$ is subcritical.
\end{conjecture}

We believe that Conjecture~\ref{con:highd} should follow from similar methods to those used in the present paper, but with significant technical complications.

Our final question concerns directed triangular bootstrap percolation, which was the example subcritical $\UU$-bootstrap percolation process given in the introduction. The lower bound in Corollary \ref{cor:DTBP} obtained by analysing our proof is likely to be far from the truth. What is the correct value of $p_c(\triD,2)$?

\begin{question}
Can one obtain better bounds on the critical probability $p_c(\triD,2)$ for DTBP than those given in Corollary \ref{cor:DTBP}?
\end{question}

Finally we remark that there are many other interesting questions that one could and should ask about subcritical $\UU$-bootstrap percolation -- too many to list here individually.

\bibliographystyle{amsplain}

 \bibliography{mylargebib}

\end{document}